\documentclass{amsart}
\usepackage{amssymb}
\usepackage{mathrsfs}
\usepackage{graphicx}
\usepackage{url}

\newtheorem{theorem}{Theorem}[section]
\newtheorem{proposition}[theorem]{Proposition}
\newtheorem{corollary}[theorem]{Corollary}
\newtheorem{lemma}[theorem]{Lemma}
\theoremstyle{remark}
\newtheorem{definition}[theorem]{Definition}
\newtheorem{remark}[theorem]{Remark}
\newtheorem{example}[theorem]{Example}
\newtheorem{problem}[theorem]{Problem}
\newtheorem{assumptions}[theorem]{Assumptions}
\newtheorem{conjecture}[theorem]{Conjecture}

\numberwithin{equation}{section}

\begin{document}

\begin{abstract}
We introduce the study of frames and equiangular lines in classical geometries over finite fields.
After developing the basic theory, we give several examples and demonstrate finite field analogs of equiangular tight frames (ETFs) produced by modular difference sets, and by translation and modulation operators.
Using the latter, we prove that Gerzon's bound is attained in each unitary geometry of dimension $d = 2^{2l+1}$ over the field $\mathbb{F}_{3^2}$.
We also investigate interactions between complex ETFs and those in finite unitary geometries, and we show that every complex ETF implies the existence of ETFs with the same size over infinitely many finite fields.
\end{abstract}

\title[Frames over finite fields]{Frames over finite fields: Basic theory and equiangular lines in unitary geometry}
\author{Gary R.~W.~Greaves}
\address{School of Physical and Mathematical Sciences, Nanyang Technological University, Singapore 637371}
\email{gary@ntu.edu.sg}
\author{Joseph W.~Iverson}
\address{Department of Mathematics, Iowa State University, Ames, IA USA 50010}
\email{jwi@iastate.edu}
\author{John Jasper}
\address{Department of Mathematics and Statistics, South Dakota State University, Brookings, SD USA 57007}
\email{john.jasper@sdstate.edu}
\author{Dustin G.~Mixon}
\address{Department of Mathematics, The Ohio State University, Columbus, OH USA 43210}
\email{mixon.23@osu.edu}
\maketitle

\section{Introduction}

How many equiangular lines through the origin can be packed in $d$-dimensional space?
This basic yet challenging geometric problem has been an active area of research for at least 50 years, and in that time researchers have identified fundamental connections with such diverse areas as algebraic combinatorics, data science, and quantum information theory.
The maximum number of equiangular lines in $d$-dimensional space depends on the underlying base field, and before now research has focused on the real numbers and their extensions (usually just the real and complex numbers, but quaternions and octonions are also studied~\cite{CKM16,ET20,W20}).
In this paper and its companion~\cite{FFF2}, we initiate a study of equiangular lines over finite fields.

To further elucidate, take $\mathbb{F} \in \{\mathbb{R},\mathbb{C}\}$ and equip $\mathbb{F}^d$ with the usual inner product $(\cdot,\cdot)$.
Throughout, we abbreviate $[n] := \{1,\dotsc,n\}$.
A \emph{line} in $\mathbb{F}^d$ is simply a one-dimensional subspace $\ell \leq \mathbb{F}^d$, and a sequence $\{ \ell_j \}_{j \in [n]}$ of lines is said to be \emph{equiangular} with parameter $b \in [0,1)$ if unit-norm representatives $\varphi_j \in \ell_j$ satisfy $| ( \varphi_i, \varphi_j ) |^2 = b$ whenever $i \neq j$.
Here, $b$ represents the square of the cosine of the angle between $\ell_i$ and $\ell_j$.
An upper bound normally attributed to Michael Gerzon states that $n$ equiangular lines exist in $\mathbb{F}^d$ only if $n \leq \frac{d(d+1)}{2}$ for $\mathbb{F} = \mathbb{R}$, and only if $n \leq d^2$ for $\mathbb{F} = \mathbb{C}$~\cite{LS73}.
In the real setting, \textbf{Gerzon's bound} is saturated for $d \in \{2,3,7,23\}$, but not for any other $d < 89$.
For example, $n=3$ equiangular lines in $\mathbb{R}^2$ are spanned by vectors that resemble the Mercedes--Benz logo, and $n=6$ equiangular lines in $\mathbb{R}^3$ connect antipodal vertices of a 20-sided die (regular icosahedron).
In contrast with the real case, the complex version of Gerzon's bound is conjectured to be achieved in every dimension:
\begin{conjecture}
\label{conj: Zauner complex}
For every $d \geq 2$, there exist $d^2$ equiangular lines in $\mathbb{C}^d$.
\end{conjecture}
This is known as \textbf{Zauner's conjecture}, since an equivalent conjecture first appeared in his 1999 PhD thesis~\cite{Z99}.\footnote{Elsewhere in the literature, ``Zauner's conjecture'' 
may refer either to Conjecture~\ref{conj: Zauner complex} or to one of several stronger statements that posit the existence of equiangular lines with additional structure.
We do not have any additional structure in mind when we use this term.}
It is a basic problem in quantum information theory, and $d^2$ orthogonal projections onto equiangular lines in $\mathbb{C}^d$ form a particularly nice basis for the real space of $d\times d$ self-adjoint matrices.
For this reason and others, Zauner's conjecture has been the subject of intense interest since its introduction over 20 years ago, and recently a monetary prize was announced for its resolution~\cite{HRZ20}.
Nevertheless, the problem remains wide open, and at present there are only finitely many dimensions $d$ for which $d^2$ equiangular lines are known to exist in $\mathbb{C}^d$.

More is known about the maximum number of equiangular lines with a given angle.
The \textbf{relative bound} states that $n$ equiangular lines with parameter $b \in [0,\frac{1}{d})$ exist in $\mathbb{F}^d$ only if $n \leq \frac{d(1-b)}{1-bd}$.
Conditions for equality are understood in terms of \emph{frame theory}.
To elaborate, let $\{ \varphi_j \}_{j \in [n]}$ be a sequence of column vectors in $\mathbb{C}^d$, expressed as a $d\times n$ matrix $\Phi = \begin{bmatrix} \varphi_1 & \dotso & \varphi_n \end{bmatrix}$.
We call $\Phi$ a \emph{tight frame} with constant $c>0$ if $\Phi \Phi^* = c I$; an \emph{equiangular tight frame} (ETF) is a tight frame consisting of equal-norm representatives for equiangular lines.
The matrix representing a tight frame is optimally conditioned, and the columns of a ${d \times n}$ ETF have optimal coherence:
they span $n$ lines in $\mathbb{F}^d$ whose sharpest (acute) angle is as wide as possible.
These features and others make ETFs suitable for applications in areas such as wireless communication~\cite{SH03},
compressed sensing~\cite{BFMW13}, digital fingerprinting~\cite{MQKF13}, and quantum information theory~\cite{RBSC04}.
It is known that a sequence of equiangular lines $\{ \ell_j \}_{j\in [n]}$ with parameter $b$ in $\mathbb{F}^d$ saturates the relative bound if and only if equal-norm representatives $\varphi_j \in \ell_j$ form the columns of an ETF.
Furthermore, any sequence of equiangular lines that attains Gerzon's bound also creates an ETF by taking equal-norm representatives, so that Zauner's conjecture is equivalent to the existence of a $d \times d^2$ complex ETF for every $d \geq 2$.
Beyond Zauner's conjecture we have the more general (and similarly formidable) existence problem of complex ETFs, for which cash prizes are also available~\cite{M:SFM1,M:SFM2}.

\begin{problem}
\label{prob: complex ETF existence}
For which pairs $(d,n)$ does a complex $d\times n$ ETF exist?
\end{problem}

A variety of complex ETF constructions are known, including many infinite families of sizes $(d,n)$ with $n < d^2$~\cite{FM:T,CGSZ16,FJMP18,FJMPW17,FJM16,FMT12,IJM:FGA,IJM:NAG}.
However, very little is understood about nonexistence, and the known list of necessary conditions is as short as this: $n \leq d^2$, $n \leq (n-d)^2$, and $(d,n) \notin{\{ (3,8),} {(5,8) \}}$.
Apart from Gerzon's bound, $3 \times 8$ and $5 \times 8$ are the only sizes for which a complex ETF is known not to exist, and this was proven by means of a Gr\"{o}bner basis calculation that does not appear to be feasible for larger sizes~\cite{Sz14}.
This state of affairs is particularly grim when compared with the case of real ETFs, for which a great many nonexistence results are known, a success that may be due to the fundamentally discrete nature of the sign pattern in the Gram matrix $\Phi^\top \Phi$ of a real ETF~\cite{STDH07,LS73,FM:T}.

Faced with slow progress on problems of import, we take here the advice of Polya and \emph{vary the problem}~\cite{P14}.
In particular, we observe that each property of interest (equiangular lines, tight frame, ETF) may be expressed using only polynomials and an order-2 field automorphism (complex conjugation).
From this light, our problems are fundamentally algebraic, and they readily admit generalizations over any field, or even $\ast$-ring.

In this article and its companion~\cite{FFF2} we focus on ETFs over finite fields.
Here we join a long tradition of investigating finite field analogs of problems originally posed over real and complex numbers.
Among the many examples, we mention 
the local/global principle of number theory~\cite{Ma93}, 
the Ax--Grothendieck theorem~\cite{Grot66},
sum--product estimates~\cite{BKT04},
the Erd\H{o}s--Falconer distance problem~\cite{IR07},
the Kakeya problem~\cite{Dv09},
and Roth's theorem on arithmetic progressions~\cite{Cr17,EG17}.
For each of these examples, the generalization to finite fields proved to be fruitful, either by producing insight or traction for the original problem, or by creating an alternate arena that demonstrated interest in its own right.
We also find a precursor within frame theory, in the work of Bodmann et al. on frames over $\mathbb{F}_2$ and their relationship with coding theory
(a subject and viewpoint not represented here, since we do not concern ourselves with Hamming distance)~\cite{BLRTT09}.

Broadly speaking, when varying a problem one hopes to balance two competing goals.
On the one hand, the new problem should have enough features in common that it could reasonably provide some insight for the original one.
On the other hand, the new problem should have some strikingly different features that provide new openings for attack.

Both of these objectives are met for ETFs over finite fields.
The new problem has the same linear algebraic expression as the old one, and we are able to derive a basic theory that includes parallels of many of the most important results, including factorization of Gram matrices (Theorems~\ref{thm: factor U} and~\ref{thm: factor O}) and Gerzon's bound (Theorem~\ref{thm: Gerzon}).
Furthermore, we show that the existence of a $d \times n$ complex ETF implies the existence of $d \times n$ ETFs over infinitely many finite fields (Theorem~\ref{thm: C to F}); in this sense the finite field problem is properly a generalization of the complex one.
In the case of finite orthogonal geometries, ETFs over finite fields sometimes imply the existence of real ETFs~\cite{FFF2}.
Meanwhile, the finite field setting provides new features and new tools, most notably the existence of isotropy (nonzero vectors being orthogonal to themselves) and the absence of norm positivity.
Furthermore, the finite vector spaces under consideration provide opportunities for new computational methods (such as exhaustive search).
Finally, the discrete nature of finite fields suggests comparison with real ETFs, which are understood much better than their complex counterparts.

Thanks to these differences, we are able to prove for finite fields what has eluded researchers in the complex setting for more than 20 years: we prove that Gerzon's bound is attained in infinitely many dimensions.
Specifically, for unitary geometries over $\mathbb{F}_{3^2}$ we use translations and modulations to produce a $d \times d^2$ ETF whenever ${d = 2^{2l+1}}$ (Theorem~\ref{thm: Zauner}).
Our construction may be viewed as generalizing Hoggar's lines to an infinite family over a finite field (Remark~\ref{rem: Hoggar}); in particular, our construction leverages an irreducible representation of the Heisenberg group over the elementary abelian group $\mathbb{Z}_2^{2l+1}$, as opposed to the cyclic group $\mathbb{Z}_{2^{2l+1}}$.
Similarly, in the companion paper~\cite{FFF2} we prove that Gerzon's bound is attained in a finite orthogonal geometry of dimension $d$ whenever $|d-7|$ is not a power of~2.
Overall, finite fields provide a fruitful environment in which to study ETFs, and due to their rich theory they appear to be worthy of study in their own right.

This paper is organized in two parts.
Part~\ref{part:1} develops the basic theory necessary for a rigorous investigation in the sequel and in the companion paper~\cite{FFF2}.
Throughout Part~\ref{part:1} we treat both orthogonal and unitary geometries over finite fields, corresponding to real and complex Hilbert spaces, respectively.
We begin with a review of the theory of forms in Section~\ref{sec: prelim}, followed by an exposition of frame theory over finite fields in Section~\ref{sec: frame theory}.
Section~\ref{sec: equiangular} treats equiangular lines and culminates in Gerzon's bound (Theorem~\ref{thm: Gerzon}).

Part~\ref{part:2} focuses on ETFs in unitary geometries, which are the finite field analog of complex Hilbert spaces.
In Section~\ref{sec: first ex}, we give our first examples of such ETFs, and we show that a generalization of difference sets creates ETFs just as in the complex setting (Theorem~\ref{thm: harm diff}).
Section~\ref{sec: Gabor} develops the theory of Gabor frames over finite fields.
After proving that every Gabor frame is tight (Proposition~\ref{prop: Gabor tight frame}), we explicitly describe fiducial vectors that create $d \times d^2$ ETFs in infinitely many dimensions over $\mathbb{F}_{3^2}$ (Theorem~\ref{thm: Zauner}).
Finally, Section~\ref{sec: C to F} demonstrates connections with the complex setting.
We first prove that when a complex $d\times n$ ETF exists, there is also a complex $d\times n$ ETF with algebraic entries (Theorem~\ref{thm:alg}).
Then we show that every algebraic ETF projects into infinitely many finite fields (Theorem~\ref{thm:project}), so that every complex ETF implies a multitude of finite field ETFs having the same size (Theorem~\ref{thm: C to F}).
The paper ends with some open problems for future research.
(Additional open problems are scattered throughout.)

\part{Basic theory}
\label{part:1}

\section{Preliminaries}
\label{sec: prelim}
For the sake of accessibility, we begin with a short review of the theory of forms.
Standard references include \cite{Gr02,T92,Ar57}.

\begin{assumptions}
Throughout Part~1, we fix a field $\mathbb{F}$ and a field automorphism $\sigma \colon \mathbb{F} \to \mathbb{F}$ (possibly the identity) that satisfies $\sigma^2 = 1$.
The subfield fixed by $\sigma$ is denoted $\mathbb{F}_0 = \{ a \in \mathbb{F} : a^\sigma = a \}$.
We also fix a vector space $V$ over $\mathbb{F}$ of finite dimension $d = \dim V$, and we assume $V$ is equipped with a form $\langle \cdot , \cdot \rangle \colon V \times V \to \mathbb{F}$ that satisfies the following conditions:
	\begin{itemize}
	\item[(F1)]
	for every $u \in V$, the induced mapping $\langle u, \cdot \rangle \colon V \to \mathbb{F}$ is linear,
	\item[(F2)]
	$\langle u, v \rangle = \langle v, u \rangle^\sigma$ for every $u,v \in V$,
	\item[(F3)]
	if $u \in V$ satisfies $\langle u,v \rangle$ for every $v \in V$, then $u = 0$.
	\end{itemize}
We write $Q \colon V \to \mathbb{F}$ for the associated \textbf{quadratic form} given by $Q(v) = \langle v , v \rangle$.
By (F1) and (F2), $\langle \cdot, \cdot \rangle$ is linear in the second variable and $\sigma$-linear in the first.
\end{assumptions}

Assumptions (F1)--(F3) may also be expressed in terms of matrices.
Specifically, for a matrix $A  = \begin{bmatrix} a_{ij} \end{bmatrix} \in \mathbb{F}^{m\times n}$ we denote $A^* = \begin{bmatrix} a_{ji}^\sigma \end{bmatrix} \in \mathbb{F}^{n\times m}$.
Choose any basis $e_1,\dotsc,e_d \in V$ with coordinate transformation $T \colon V \to \mathbb{F}^d$.
Then conditions (F1)--(F3) are equivalent to the existence of a \textbf{Gram matrix} $M \in \mathbb{F}^{d\times d}$ satisfying:
	\begin{itemize}
	\item[(F1')]
	$\langle u, v \rangle = (Tu)^* M (Tv)$ for every $u,v \in V$,
	\item[(F2')]
	$M=M^*$,
	\item[(F3')]
	$M$ is invertible.
	\end{itemize}
Explicitly, $M = \begin{bmatrix} \langle e_i, e_j \rangle \end{bmatrix} \in \mathbb{F}^{d\times d}$.

By (F3'), $\det M$ belongs to the multiplicative group $\mathbb{F}^\times$ of $\mathbb{F}$.
Throughout the paper, we write $\mathbb{F}^{\times 2} \leq \mathbb{F}^\times$ for the subgroup of quadratic residues, and in this notation the \textbf{discriminant} of $V$ is defined to be
\[
\operatorname{discr}(V) = (\det M) \mathbb{F}^{\times 2} \in \mathbb{F}^\times / \mathbb{F}^{\times 2}.
\]
It is an invariant of $\langle \cdot , \cdot \rangle$ and independent of the choice of basis behind $M$.

The following space plays the same role as $\ell^2$ in classical frame theory.

\begin{definition}
\label{def: real complex models}
We write $(\cdot, \cdot)$ for the form on $\mathbb{F}^n$ with Gram matrix $M = I$, i.e.,
\[
(x,y) = x^* y = \sum_{i \in [n]} x_i^\sigma y_i
\quad
\text{for }x = \{ x_i \}_{i \in [n]}, \, y = \{ y_i \}_{i \in [n]} \in \mathbb{F}^n.
\]
(Notice that we take $\mathbb{F}^n$ to consist of column vectors, and that $(\cdot, \cdot)$ is conjugate-linear in the first variable.)
If $\sigma = 1$, then we refer to $\mathbb{F}^n$ equipped with $(\cdot, \cdot)$ as a \textbf{real model}, since it is reminiscent of the real Hilbert space $\mathbb{R}^n$.
If $\sigma \neq 1$, then we call it a \textbf{complex model}.
\end{definition}

Next, let $W \leq V$ be a subspace.
Its \textbf{orthogonal complement} is the subspace
\[
W^\perp = \{ u \in V : \langle w, u \rangle = 0 \text{ for every }w \in W\},
\]
which satisfies $\dim W + \dim W^\perp = V$.
We say $W$ is \textbf{nondegenerate} if $W \cap W^\perp = \{0\}$; equivalently, the restriction of $\langle \cdot, \cdot \rangle$ to $W$ again satisfies (F1)--(F3).
At the other extreme, $W$ is called \textbf{totally isotropic} if $W \leq W^\perp$, that is, $\langle u,v \rangle = 0$ for every $u,v \in W$.
In that case, $\dim W \leq \tfrac{1}{2} \dim V$.

Now let $W$ be another vector space over $\mathbb{F}$ with a form satisfying (F1)--(F3).
Every linear map $A \colon W \to V$ has a unique \textbf{adjoint} $A^\dagger \colon V \to W$ satisfying $\langle A u, v \rangle = \langle u, A^\dagger v \rangle$ for every $u \in W$ and $v \in V$.
Its kernel is $\operatorname{Ker} A^\dagger = (\operatorname{Im} A)^\perp$, and in particular, $A^\dagger$ is one-to-one if and only if $A$ is onto (and vice versa).
Adjoints are related to but distinct from conjugate transpose matrices, and a matrix for $A^\dagger$ can be found as follows.
After choosing bases, we may assume there are matrices $M$ and $N$ such that $V = \mathbb{F}^m$ has form $\langle u, v \rangle = u^* M v$ and $W = \mathbb{F}^n$ has form $\langle x, y \rangle = x^* N y$.
If we identify operators between $\mathbb{F}^n$ and $\mathbb{F}^m$ with matrices, then the adjoint of $A \in \mathbb{F}^{m\times n}$ is $A^\dagger = N^{-1} A^* M \in \mathbb{F}^{n \times m}$.

We define the \textbf{isometry group} of $V$ to be the group $\operatorname{I}(V) \leq \operatorname{GL}(V)$ of operators $A \colon V \to V$ satisfying $A^\dagger A = I$.
More generally, $\Delta(V) \leq \operatorname{GL}(V)$ is the group of all $A$ satisfying $A^\dagger A = c I$ for $c \in \mathbb{F}^\times$.
The distinction between the two groups is necessary: every scalar multiple of an isometry belongs to $\Delta(V)$, but they may form a proper subgroup.

We assume no more than (F1)--(F3) in general, but we will be especially interested in the following special cases, which represent classical geometries over finite fields.

\begin{definition}
\label{def: Case OU}
We say we are in \textbf{Case O} if $\mathbb{F} = \mathbb{F}_q$ is a finite field of odd order $q$, and $\sigma = 1$.
On the other hand, \textbf{Case U} occurs when $\mathbb{F} = \mathbb{F}_{q^2}$ is finite (possibly of even order) and $\sigma$ is given by $\alpha^\sigma = \alpha^q$.
\end{definition}

In both Case O and Case U, $q$ is necessarily a prime power since it is the size of a finite field.
For Case O, we choose to specify that $q$ is odd since the theory of quadratic forms in even characteristic is fundamentally different.
In reference to Case~U, we remark that $\alpha^\sigma = \alpha^q$ describes the only nontrivial field automorphism on $\mathbb{F}_{q^2}$ that satisfies $\sigma^2 = 1$.

\smallskip

Suppose for the moment that Case~O occurs.
Then $\mathbb{F}_0 = \mathbb{F}$, and $\langle \cdot, \cdot \rangle$ is a nondegenerate symmetric bilinear form.
We call $V$ a \textbf{quadratic space} and say it has an \textbf{orthogonal geometry}.
Since $q$ is odd, $\langle \cdot, \cdot \rangle$ can be recovered from $Q$ through the polarization identity
\[
\langle u, v \rangle = \tfrac{1}{2} \bigl[ Q(u+v) - Q(u-v) \bigr],
\quad \text{for }u,v \in V.
\]
Up to isomorphism there are exactly two possibilities for $\langle \cdot, \cdot \rangle$, corresponding to whether or not $\operatorname{discr}(V) \in \mathbb{F}^\times / \mathbb{F}^{\times 2}$ is trivial.
Given another quadratic space $W$ over $\mathbb{F}_q$, there exists an isomorphism $V \to W$ that preserves the form if and only if $\dim V = \dim W$ and $\operatorname{discr}(V) = \operatorname{discr}(W)$.
In particular, given an invertible matrix $M \in \mathbb{F}^{d\times d}$, there exists a basis for $V$ having matrix $M$ as a Gram matrix if and only if $(\det M)\mathbb{F}^{\times 2} = \operatorname{discr}(V)$.
Thus, $V$ admits an orthonormal basis if and only if $\operatorname{discr}(V)$ is trivial, if and only if $V$ is isometrically isomorphic with the real model on $\mathbb{F}_q^d$.
The isometry group is denoted $\operatorname{I}(V) = \operatorname{O}(V)$ since it is an example of a classical \textbf{orthogonal group}.
If $d = \dim V$ is even, then $\operatorname{O}(V) \times \mathbb{F}_q^\times$ is a proper subgroup of $\Delta(V)$ since there exist operators satisfying $A^\dagger A = \alpha I$, where $\alpha \in \mathbb{F}^\times$ is not a quadratic residue.
If $d$ is odd then the isomorphism type of $\operatorname{O}(V)$ does not depend on $\operatorname{discr}(V)$, and $\Delta(V) = \operatorname{O}(V) \times \mathbb{F}_q^\times$ consists of nonzero scalar multiples of orthogonal matrices.
\smallskip

Now suppose we are in Case U.
Then $\langle \cdot, \cdot \rangle$ is a nondegenerate Hermitian form, and $(V,\langle \cdot, \cdot \rangle)$ is isometrically isomorphic with the complex model on $\mathbb{F}_{q^2}^d$.
We call $V$ a \textbf{unitary space} and say it has a \textbf{unitary geometry}.
The subfield fixed by $\sigma$ is $\mathbb{F}_0 = \mathbb{F}_q \leq \mathbb{F}_{q^2}$.
We denote
\[
\mathbb{T}_q = \{ \alpha \in \mathbb{F}_{q^2} : \alpha^q \alpha = 1 \} \leq \mathbb{F}_{q^2}^\times,
\]
with $|\mathbb{T}_q| = q+1$.
If $\omega \in \mathbb{T}_q$ is a generator then we have the polarization identity
\[
\langle u, v \rangle = \frac{1}{q+1} \sum_{k=1}^{q+1} \omega^{-k} Q(u + \omega^k v),
\quad \text{for }u,v \in V.
\]
An orthonormal basis always exists, so (F1)--(F3) determine $\langle \cdot, \cdot \rangle$ uniquely up to isometric isomorphism.
The isometry group is denoted $I(V) = \operatorname{U}(V)$ since it is an example of a classical \textbf{unitary group}.
For the special case of $V = \mathbb{F}_{q^2}^d$ equipped with $(\cdot, \cdot)$ (the complex model) we also write $I(V) = \operatorname{U}(d,q)$.
Here, $\operatorname{U}(d,q)$ consists precisely of \textbf{unitaries}, i.e., matrices $U \in \mathbb{F}_{q^2}^{d\times d}$ that satisfy $U^* U = I$.
Notice that $\operatorname{U}(1,q) = \mathbb{T}_q$, which justifies the latter notation.
Here $\Delta(V) = \operatorname{U}(V) \times \mathbb{F}_{q^2}^\times$ consists of nonzero scalar multiples of unitary operators.

\begin{remark}
Many standard techniques in frame theory over $\mathbb{R}$ or $\mathbb{C}$ rely on the fact that if a matrix commutes with its conjugate transpose then it can be diagonalized by an isometric isomorphism.
The reader is warned that this theorem fails over finite fields, and alternative methods are needed in its place.
In Case~U, if $A$ is a square matrix that commutes with $A^*$, it does \emph{not} follow that $A$ is diagonalizable.
On the contrary, every $B \in \mathbb{F}_q^{n \times n}$ is similar to some $A \in \mathbb{F}_{q^2}^{n\times n}$ satisfying $A = A^*$~\cite{Gu18}.
\end{remark}

\section{Frame theory}
\label{sec: frame theory}

This section develops the basics of frame theory over arbitrary fields, with special emphasis placed on finite fields.
We prove generalizations of standard results from frame theory over $\mathbb{R}$ and $\mathbb{C}$~\cite{CKF13,W18}.
Over finite fields the main differences are as follows: the quadratic form $Q(v) = \langle v, v \rangle$ satisfies no condition akin to positivity, there are usually nonzero vectors $v \in V$ with $Q(v) = 0$, there are usually tight frames with frame constant zero, and there are two types of orthogonal geometries in Case~O.
Each of these differences has repercussions for the basic theory outlined below.

\subsection{Finite frames}

\begin{definition}
Throughout the paper we abuse notation by identifying a finite sequence $\Phi = \{ \varphi_j \}_{j\in [n]}$ in $V$ with its \textbf{synthesis operator} $\Phi \colon \mathbb{F}^n \to V$ given by
\[
\Phi \{ x_i \}_{i\in [n]} = \sum_{k=1}^n x_k \varphi_k.
\]
Its adjoint with respect to $(\cdot , \cdot)$ is called the \textbf{analysis operator} $\Phi^\dagger \colon V \to \mathbb{F}^n$ given by $\Phi^\dagger v = \{ \langle \varphi_i, v \rangle \}_{i\in [n]}$, and its \textbf{frame operator} is $\Phi \Phi^\dagger \colon V \to V$, where
\[
\Phi \Phi^\dagger v = \sum_{k=1}^n \langle \varphi_k, v \rangle \varphi_k.
\]
Multiplying in the other direction gives the \textbf{Gramian} $\Phi^\dagger \Phi \colon \mathbb{F}^n \to \mathbb{F}^n$, whose representation in the standard basis is known as the \textbf{Gram matrix} $\begin{bmatrix} \langle \varphi_i, \varphi_j \rangle \end{bmatrix}$ of $\Phi$.

We call $\Phi$ a \textbf{frame} if its vectors span $V$.
It is \textbf{nondegenerate} if its frame operator is invertible.
Since $d = \dim V$, a frame $\Phi$ with $n$ vectors is said to have \textbf{size} $d\times n$ (matching the size of a matrix for its synthesis operator).
A \textbf{tight frame} for $V$ is defined as a frame $\Phi$ that satisfies $\Phi \Phi^\dagger = c I$ for $c \in \mathbb{F}$.
Here $c$ is known as the \textbf{frame constant}, and we also call $\Phi$ a \textbf{$c$-tight frame}.
If $c=1$ the frame is called \textbf{Parseval}.
If $c = 0$ it is \textbf{totally isotropic}.
We emphasize that the vectors of a totally isotropic tight frame must span $V$, and it is not sufficient that $\Phi \Phi^\dagger = 0$.
\end{definition}

If $V = \mathbb{F}^d$ then we perform a further abuse by identifying $\Phi$ with the matrix of its synthesis operator, which is the $d\times n$ matrix whose $j$-th column is $\varphi_j$.
We also identify the analysis operator with its matrix, and if the form on $V$ is given by $\langle u, v \rangle = u^* M v$ then $\Phi^\dagger = \Phi^* M$.

\begin{example}
Totally isotropic (hence degenerate) frames exist, as shown by the simple example $\Phi = \left[ \begin{array}{ccc} 1 & 1 & 1 \end{array} \right]$ in the real model over $\mathbb{F}_3$.
\end{example}

\begin{example}
\label{ex:MB}
Let $q$ be an odd prime power, and consider $\mathbb{F}_q^2$ in the real model.
If $3 \in \mathbb{F}_q^{\times 2}$ then we can choose $\beta \in \mathbb{F}_q^\times$ satisfying $\beta^2 = 3$ to obtain the so-called Mercedes--Benz frame
\[
\Phi = 
\tfrac{1}{2}
\left[
\begin{array}{ccc}
0 & - \beta & \beta \\
2 & -1 & -1
\end{array}
\right].
\]
It is tight with constant $c=3/2$, and its Gram matrix is
\begin{equation}
\label{eq:MBG}
\Phi^* \Phi = 
\tfrac{1}{2}
\left[
\begin{array}{rrr}
2 & -1 & -1 \\
-1 & 2 & -1 \\
-1 & -1 & 2
\end{array}
\right].
\end{equation}
Conversely, if $3 \notin \mathbb{F}_q^{\times 2}$ then there does not exist a frame $\Phi \in \mathbb{F}_q^{2 \times 3}$ in the real model having Gram matrix~\eqref{eq:MBG}.
This is a consequence of Theorem~\ref{thm: factor O} below.

The complex model on $\mathbb{F}_{q^2}^2$ is more permissive, as a tight frame $\Phi \in \mathbb{F}_{q^2}^{2\times 3}$ with Gram matrix~\eqref{eq:MBG} exists provided $3 \nmid q$.
This follows from Theorem~\ref{thm: factor U} below, and an explicit example appears in Example~\ref{ex:unitaries}.
\end{example}

As in the real and complex settings, frames are characterized by an expansion property involving a \textbf{dual frame} ($\Psi$ below).
If $\Phi$ is a $c$-tight frame and $c \neq 0$, then $\Psi = \frac{1}{c} \Phi$ is a dual frame.
There does not appear to be such a canonical choice of dual for $0$-tight frames.

\begin{proposition}
A sequence $\Phi = \{ \varphi_j \}_{j\in [n]}$ in $V$ is a frame if and only if there is another sequence $\Psi = \{ \psi_j \}_{j\in [n]}$ such that $\Phi \Psi^\dagger = I$, that is,
\[
\sum_{j\in [n]} \langle \psi_j, v \rangle \varphi_j = v
\quad
\text{for every }
v \in V.
\]
\end{proposition}

\begin{proof}
We may assume $V = \mathbb{F}^d$ with form $\langle u, v \rangle = u^* M v$.
The reverse implication is clear.
Conversely, if $\Phi \in \mathbb{F}^{d\times n}$ is a frame then there exists $A \in \mathbb{F}^{n\times d}$ such that $\Phi A = I$.
Put $\Psi = (A M^{-1})^*$ to obtain $\Phi \Psi^\dagger = I$.
\end{proof}

\begin{proposition}
\label{prop: tight frame}
If $\Phi = \{ \varphi_j \}_{j\in [n]}$ is a frame for $V$, then the following are equivalent for any choice of $c \in \mathbb{F}$:
	\begin{itemize}
	\item[(i)]
	$\Phi$ is a $c$-tight frame,
	\item[(ii)]
	$(\Phi^\dagger \Phi)^2 = c ( \Phi^\dagger \Phi)$,
	\item[(iii)]
	$(\Phi^\dagger u, \Phi^\dagger v) = c \langle u, v \rangle$ for every $u,v \in V$.
	\end{itemize}
Moreover, if $V = \mathbb{F}^d$ and $\langle \cdot, \cdot \rangle = ( \cdot, \cdot)$, then (i)---(iii) are equivalent to:
	\begin{itemize}
	\item[(iv)]
	the rows $\psi_1,\dotsc,\psi_d \in \mathbb{F}^n$ of $\Phi$ satisfy $( \psi_i, \psi_j) = c \delta_{i,j}$ for every $i,j \in [d]$.
	\end{itemize}
(Here and throughout, $\delta_{ij}$ is the Kronecker delta function.)
\end{proposition}

\begin{remark}
In Case O or Case U (Definition~\ref{def: Case OU}) we may add another equivalent condition to Proposition~\ref{prop: tight frame}:
\begin{itemize}
\item[(v)]
$(\Phi^\dagger v, \Phi^\dagger v) = c \cdot Q(v)$ for every $v \in V$.
\end{itemize}
Indeed, (v) is equivalent to (iii) since $\langle \cdot, \cdot \rangle$ and $(\cdot, \cdot)$ admit polarization identities.
\end{remark}

The techniques in the following argument are standard, but we include a proof for the sake of completeness.

\begin{proof}[Proof of Proposition~\ref{prop: tight frame}]
If (i) holds, then (ii) follows since
\[
(\Phi^\dagger \Phi)^2 
= \Phi^\dagger(\Phi \Phi^\dagger) \Phi 
= c \Phi^\dagger \Phi.
\]

To see that (ii) implies (iii), assume (ii) holds and choose $u,v \in V$ arbitrarily.
Since $\Phi$ is a frame, there exist $x,y \in \mathbb{F}^n$ such that $u = \Phi x$ and $v = \Phi y$.
Using the fact that $\Phi^\dagger \Phi$ is self-adjoint, we deduce that
\[
( \Phi^\dagger u, \Phi^\dagger v)
= (\Phi^\dagger \Phi x, \Phi^\dagger \Phi y)
= ( (\Phi^\dagger \Phi)^2 x, y )
= c (\Phi^\dagger \Phi x, y)
= c \langle \Phi x, \Phi y \rangle
= c \langle u, v \rangle.
\]

Next, suppose (iii) holds.
For any choice of $u,v \in V$ we have
\[
0 
= (\Phi^\dagger u, \Phi^\dagger v) - c \langle u, v \rangle
= \langle \Phi \Phi^\dagger u, v \rangle - \langle cu, v \rangle
= \langle (\Phi \Phi^\dagger - c I)u, v \rangle.
\]
Since $\langle \cdot, \cdot \rangle$ satisfies (F3), it follows that $(\Phi \Phi^\dagger - cI)u = 0$ for every $u \in V$, that is, $\Phi \Phi^\dagger = c I$.
By assumption, $\Phi$ is a frame, and so (i) follows.

Finally, assume that $V = \mathbb{F}^n$ and $\langle \cdot, \cdot \rangle = ( \cdot, \cdot )$.
Recall that in this case, we identify the synthesis operator $\Phi$ with the matrix $\Phi = \begin{bmatrix} \psi_1 \\ \vdots \\ \psi_d \end{bmatrix}$, and we identify the analysis operator $\Phi^\dagger$ with the matrix $\Phi^* = \begin{bmatrix} \psi_1^* & \cdots & \psi_d^* \end{bmatrix}$.
Then $\Phi \Phi^\dagger$ is the matrix $\Phi\Phi^* = \begin{bmatrix} \psi_i^* \psi_j \end{bmatrix}_{i,j \in [d]} = \begin{bmatrix} ( \psi_i, \psi_j) \end{bmatrix}_{i,j \in [d]}$.
Since $\Phi$ is assumed to be a frame, (i) holds if and only if $cI = \Phi \Phi^* = \begin{bmatrix} ( \psi_i, \psi_j) \end{bmatrix}_{i,j \in [d]}$, if and only if (iv) holds.
\end{proof}

\begin{remark}
In Proposition~\ref{prop: tight frame}, the hypothesis that $\Phi$ is a frame for $V$ cannot be removed.
This is familiar from the real and complex settings, where it is possible that (ii) holds and (i) fails when $\Phi$ is not a frame.
In the finite field setting, when $c = 0$ it may also happen that (iii) holds and (i) fails when $\Phi \neq 0$ is not a frame.
As an example of this, take $V = \mathbb{F}_3^3$ in the real model, and define $\Phi = \begin{bmatrix} 1 & 1 & 1 \\ 1 & 1 & 1 \end{bmatrix}$.
Then $\Phi \Phi^\dagger = 0$, and it follows easily that (ii) and (iii) hold with $c=0$.
However, (i) fails since $\Phi$ is not a frame.
\end{remark}

\begin{example}
\label{ex:unitaries}
Assume Case U (Definition~\ref{def: Case OU}).
Let $U = \begin{bmatrix} u_{ij} \end{bmatrix}_{i,j \in [n]}$ be a unitary matrix, and choose $d$ rows labeled by $J \subseteq [n]$.
Then Proposition~\ref{prop: tight frame}(iv) implies that the submatrix $\Phi = \begin{bmatrix} u_{ij} \end{bmatrix}_{i \in J, j \in [n]}$ is a $d \times n$ Parseval frame.
This gives a large supply of tight frames, and in fact every Parseval frame arises this way, as a consequence of Proposition~\ref{prop: Naimark U} below.

As a concrete example, consider the complex model over $\mathbb{F}_{5^2}$.
Let $\alpha \in \mathbb{F}_{5^2}^\times$ be a primitive element, and put $\omega = \alpha^8 \in \mathbb{T}_5$.
Then $\omega^3 = 1$, and $U = \alpha \begin{bmatrix} \omega^{ij} \end{bmatrix}_{i,j \in \mathbb{Z}/3\mathbb{Z}}$ is unitary.
Rescaling the rows labeled by $J = \{ 1, 2 \}$ gives the tight frame
\[
\Phi =
\alpha^3
\left[
\begin{array}{ccc}
1 & \omega & \omega^2 \\
1 & \omega^2 & \omega
\end{array}
\right]
\]
with Gram matrix~\eqref{eq:MBG}.
This is the Mercedes--Benz frame over $\mathbb{F}_{5^2}$, cf.\ Example~\ref{ex:MB}.
Theorem~\ref{thm: factor O} below implies that there does not exist any unitary $W \in \operatorname{U}(2,5)$ for which $W \Phi$ has entries in the subfield $\mathbb{F}_5$.
In other words we cannot ``rotate'' $\Phi$ to obtain a version of the Mercedes--Benz frame in the real model over $\mathbb{F}_5$.
\end{example}

\begin{corollary}
\label{cor: iso bound}
A frame $\Phi = \{ \varphi_j \}_{j\in [n]}$ for $V$ is totally isotropic if and only if $\operatorname{Im} \Phi^\dagger$ is a totally isotropic subspace of $\mathbb{F}^n$.
Hence, $V$ admits a $0$-tight frame of $n$ vectors only if $n \geq 2d$.
\end{corollary}

The identity $(\Phi^\dagger u, \Phi^\dagger v) = \langle \Phi \Phi^\dagger u, v \rangle$ easily implies the following parallel characterization of nondegenerate frames.

\begin{proposition}
\label{prop: nondegenerate frame}
A frame $\Phi = \{ \varphi_j \}_{j\in [n]}$ for $V$ is nondegenerate if and only if $\operatorname{Im} \Phi^\dagger$ is a nondegenerate subspace of $\mathbb{F}^n$.
\end{proposition}

\subsection{Frames from Gram matrices}

\subsubsection{Uniqueness}

\begin{proposition}
\label{prop: ker Gram}
If $\Phi$ is a frame, then $\operatorname{Ker} \Phi^\dagger \Phi = \operatorname{Ker} \Phi$ and $\operatorname{Im} \Phi^\dagger \Phi = \operatorname{Im} \Phi^\dagger$.
\end{proposition}

\begin{proof}
We show $\operatorname{Ker} \Phi^\dagger \Phi = \operatorname{Ker} \Phi$, and $\operatorname{Im}  \Phi^\dagger \Phi = \operatorname{Im} \Phi^\dagger$ follows by taking orthogonal complements.
Given $x \in \mathbb{F}^n$, we have $\Phi^\dagger \Phi x = 0$ if and only if $\langle \Phi x, \Phi y \rangle = 0$ for every $y \in \mathbb{F}^n$.
Since $\Phi$ is a frame, this happens if and only if $\Phi x = 0$.
\end{proof}

\begin{proposition}
\label{prop: unitary Gram}
Let $\Phi$ and $\Psi$ be frames for $V$ with the same number of vectors.
Given $c \in \mathbb{F}^\times$, we have $\Psi^\dagger \Psi = c \Phi^\dagger \Phi$ if and only if $\Psi = A \Phi$ for unique $A \in \Delta(V)$ satisfying $A^\dagger A = c I$.
\end{proposition}

\begin{proof}
The reverse implication is trivial.
Conversely, if $\Psi^\dagger \Psi = c \Phi^\dagger \Phi$ then the expression $A( \sum_{j \in [n]} x_j \varphi_j ) = \sum_{j\in [n]} x_j \psi_j$ gives a well-defined linear operator $A \colon V \to V$, and it is clear that $\langle A u , A v \rangle = c \langle u, v \rangle$ for every $u,v \in V$.
The choice of $A$ is unique since $\Phi$ and $\Psi$ are frames.
\end{proof}

\subsubsection{Existence}

The following was observed in Lemma~2.3 of~\cite{Gu18}.
Our proof below is more direct and gives an explicit algorithm.

\begin{theorem}
\label{thm: factor U}
Suppose we are in Case $U$ (Definition~\ref{def: Case OU}).
Then $G \in \mathbb{F}_{q^2}^{n\times n}$ is the Gram matrix of a frame for $V$ if and only if $G = G^*$ and $\operatorname{rank} G = \dim V$.
\end{theorem}

Notice that there is no condition akin to positive semidefiniteness in Theorem~\ref{thm: factor U}.
This is a significant departure from the real and complex settings.

\begin{proof}
We may assume $V = \mathbb{F}_{q^2}^d$ in the complex model.
The forward implication is clear from Proposition~\ref{prop: ker Gram}.
For the converse, we first construct $\Psi = \begin{bmatrix} \psi_1 & \cdots & \psi_n \end{bmatrix} \in \mathbb{F}_{q^2}^{2n \times n}$ such that $\Psi^* \Psi = G$.
It will take the form $\Psi = \begin{bmatrix} A \\ B \end{bmatrix}$, with $A \in \mathbb{F}_{q^2}^{n\times n}$ upper triangular and $B \in \mathbb{F}_{q^2}^{n\times n}$ a diagonal matrix.
Later, we will transform $\Psi$ to produce a frame $\Phi \in \mathbb{F}_{q^2}^{d \times n}$.

We first construct $A = \begin{bmatrix} a_1 & \cdots & a_n \end{bmatrix}$ in such a way that $(a_i, a_j) = G_{ij}$ for every $i \neq j$.
We define the entries of $A$ recursively, going down the columns from left to right.
To begin, set $a_1$ to be the first column of the $n\times n$ identity matrix.
Now suppose we have constructed the first $j-1$ columns to have the desired inner products, and that the matrix built so far is upper triangular with $1$s on the diagonal. 
We must define the entries of the next column $a_j = \begin{bmatrix} a_{kj} \end{bmatrix}_{k \in [n]}$ so that for any $i < j$,
\[
G_{ij} = (a_i, a_j) = \sum_{k=1}^n a_{ki}^q a_{kj} = a_{ij} + \sum_{k=1}^{i-1} a_{ki}^q a_{kj}.
\]
To accomplish this, we first set $a_{1j} = G_{1j}$.
After the first $i-1$ entries of $a_j$ are determined and $i < j$, we define $a_{ij} = G_{ij} - \sum_{k=1}^{i-1} a_{ki}^q a_{kj}$.
Finally, we set $a_{jj} = 1$ and $a_{ij} = 0$ for $i > j$.
Continuing in this way, we eventually obtain $A \in \mathbb{F}_{q^2}^{n\times n}$ with the desired structure.

Having built $A$, we next define the diagonal matrix $B = \begin{bmatrix} b_1 & \cdots & b_n \end{bmatrix} \in \mathbb{F}_{q^2}^{n\times n}$ with columns $b_j = \begin{bmatrix} b_{ij} \end{bmatrix}_{i \in [n]}$ in such a way that the vectors $\psi_j = \begin{bmatrix} a_j \\ b_j \end{bmatrix}$ satisfy
\[ G_{ij} = ( \psi_i, \psi_j) = (a_i, a_j) + (b_i, b_j) = \begin{cases} (a_i, a_j), & \text{if } i \neq j; \\ (a_i, a_i) + b_{ii}^{q+1}, & \text{if } i = j. \end{cases} \]
It suffices to choose $b_{ii}$ such that $b_{ii}^{q+1} = G_{ii} - (a_i, a_i)$, and this is possible since $G_{ii} - (a_i, a_i) \in \mathbb{F}_q$.
Make any valid choice to complete the construction of $B$, hence of $\Psi \in \mathbb{F}_{q^2}^{2n \times n}$ satisfying $\Psi^* \Psi = G$.

It remains to transform $\Psi \in \mathbb{F}_{q^2}^{2n \times n}$ into $\Phi \in \mathbb{F}_{q^2}^{d\times n}$.
Set $W = \operatorname{Im} \Psi \leq \mathbb{F}_{q^2}^n$, and consider its \textbf{radical}
\[
\operatorname{rad} W := \{ x \in W : \langle x, y \rangle = 0 \text{ for every }y \in W\} \leq W.
\]
Choose any algebraic complement $U \leq W$ for $\operatorname{rad} W$, that is, $W = U \oplus \operatorname{rad} W$ as vector spaces.
Then $U \leq \mathbb{F}_{q^2}^n$ is nondegenerate.
For each $j$, write $\psi_j = x_j + y_j$ with $x_j \in U$ and $y_j \in \operatorname{rad} W$.
Then $x_1,\dotsc,x_n$ span $U$ and satisfy $(x_i, x_j) = (\psi_i, \psi_j) = G_{ij}$ for every $i,j$.

Finally, we construct the matrix $\Phi$.
Set $m = \dim U$.
Since $U$ is nondegenerate, there is an isometric isomorphism $T \colon U \to \mathbb{F}_{q^2}^m$.
Define $\varphi_j = T(x_j)$ for every $j$.
Then $\Phi = \begin{bmatrix} \varphi_1 & \cdots & \varphi_n \end{bmatrix} \in \mathbb{F}_{q^2}^{m\times n}$ satisfies $\operatorname{rank} \Phi = m$ and $\Phi^* \Phi = G$.
In particular, $m = \operatorname{rank} G = d$.
This completes the proof.
\end{proof}

\begin{definition}
Given a square matrix $M  = \begin{bmatrix} M_{ij} \end{bmatrix}_{i,j\in [n]} = \begin{bmatrix} m_1 & \cdots & m_n \end{bmatrix}$, select columns $\{ m_j \}_{j\in I}$ that form a basis for $\operatorname{Im} M$.
We refer to $M_b = \begin{bmatrix} M_{ij} \end{bmatrix}_{i,j \in I}$ as a \textbf{basic submatrix} of $M$.
\end{definition}

\begin{theorem}
\label{thm: factor O}
Assume Case O (Definition~\ref{def: Case OU}).
Choose a matrix $G \in \mathbb{F}_q^{n\times n}$ and basic submatrix $G_b$.
Then, $G$ is the Gram matrix of a frame for $V$ if and only if the following hold:
	\begin{itemize}
	\item[(i)]
	$G = G^\top$,
	\item[(ii)]
	$\operatorname{rank} G = \dim V$,
	\item[(iii)]
	$(\det G_b)\mathbb{F}_q^{\times 2} = \operatorname{discr}(V)$.
	\end{itemize}
Consequently, every symmetric matrix $G \in \mathbb{F}_q^{n\times n}$ occurs as the Gram matrix of a frame for a quadratic space over $\mathbb{F}_q$, namely one in dimension $\operatorname{rank} G$ whose discriminant matches the determinant of a basic submatrix of $G$.
\end{theorem}

\begin{proof}
If $G$ is the Gram matrix of a frame $\Phi = \{ \varphi_j \}_{j\in [n]}$, then (i) holds trivially and (ii) follows from Proposition~\ref{prop: ker Gram}.
For (iii), let $g_j \in \mathbb{F}_q^n$ denote the $j$-th column of $G$ and observe that $A( \sum_{j\in [n]} x_j \varphi_j ) = \sum_{j\in [n]} x_j g_j$ gives a well-defined isomorphism $A \colon V \to \operatorname{Im} G$.
If the basic submatrix $G_b$ arises from a basis $\{ g_j \}_{j\in [n]}$ for $\operatorname{Im} G$, then the corresponding vectors $\{ \varphi_j \}_{j \in I}$ provide a basis for $V$ with Gram matrix $G_b$.
This proves (iii).

For the converse, we proceed as in the proof of Theorem~\ref{thm: factor U}.
The same argument given there constructs a matrix $A = \begin{bmatrix} a_1 & \cdots & a_n \end{bmatrix} \in \mathbb{F}_q^{n\times n}$ such that $(a_i,a_j) = g_{ij}$ for every $i \neq j$.
Next, we want to build a block-diagonal matrix $B = \begin{bmatrix} b_1 & \cdots & b_n \end{bmatrix} \in \mathbb{F}_q^{2n \times n}$ in such a way that the vectors $\psi_j = \begin{bmatrix} a_j \\ b_j \end{bmatrix}$ satisfy
\[
g_{ij} = (\psi_i, \psi_j) = (a_i, a_j) + (b_i, b_j).
\]
In other words, we want to arrange so that $(b_i, b_j) = 0$ for $ i \neq j$, while $(b_j, b_j) = g_{jj} - (a_j, a_j)$.
By Proposition~4.8 of~\cite{Gr02}, there exist vectors $x_j \in \mathbb{F}_q^2$ such that $(x_j,x_j) = g_{jj} - (a_j, a_j)$.
We let $B$ be the block-diagonal matrix $B = \operatorname{diag}(x_1,\dotsc,x_n) \in \mathbb{F}_q^{2n \times n}$.
Then $\Phi = \begin{bmatrix} A \\ B \end{bmatrix} = \begin{bmatrix} \psi_1 & \dotsc & \psi_n \end{bmatrix} \in \mathbb{F}_q^{3n \times n}$ satisfies $\Psi^\top \Psi = G$.

As in the proof of Theorem~\ref{thm: factor U}, we consider $W = \operatorname{Im} \Psi$ and an algebraic complement $U \leq W$ for its radical 
\[
\operatorname{rad} W := \{ x \in W : \langle x, y \rangle = 0 \text{ for every }y \in W\} \leq W.
\]
For each $j$ we decompose $\psi_j = x_j + y_j$ with $x_j \in U$ and $y_j \in \operatorname{rad} W$.
Then $x_1,\dotsc,x_n$ form a frame for $U$ having Gram matrix $G$.
By the forward implication already proved, $U = \operatorname{rank} G = V$  and $\operatorname{discr}(U) = (\det G_b)\mathbb{F}_q^{\times 2} = \operatorname{discr}(V)$.
Hence there is an isometric isomorphism $U \to V$, and the images of $x_1,\dotsc,x_n$ provide a frame for $V$ having Gram matrix $G$.
\end{proof}

\begin{remark}
Suppose we are in Case O (Definition~\ref{def: Case OU}) and $d = \dim V$ is odd.
If $G \in \mathbb{F}_q^{n\times n}$ is symmetric of rank $d$, then some scalar multiple $cG$ occurs as the Gram matrix of a frame for $V$.
Indeed, if $G_b$ is a basic submatrix for $G$ and $c \neq 0$, then $cG_b$ is a basic submatrix for $cG$ with determinant $c^d (\det G_b)$.
Since $d$ is odd we can choose $c$ to ensure that $(\det c G_b) \mathbb{F}_q^{\times 2} = \operatorname{discr}(V)$.
On the other hand, if $d$ is even then rescaling $G$ does not alter the existence of a frame for $V$ having $G$ as a Gram matrix.
\end{remark}

\subsection{Nondegenerate tight frames, projections, and subspaces}
The isometry group $I(V)$ has a natural action on the space of frames for $V$, and its orbits break the set of all frames into equivalence classes.
As expected, the classes for tight frames correspond with certain projections and subspaces.

\begin{proposition}
\label{prop: tight frames subspaces U}
Assume Case U (Definition~\ref{def: Case OU}).
Choose $c \in \mathbb{F}_q^\times$ and $n \geq d = \dim V$.
Then all of the following sets have the same cardinality:
	\begin{itemize}
	\item[(i)]
	$\mathscr{F} = 
	\{\text{unitary equivalence classes of $c$-tight frames for $V$ with $n$ vectors} \}$,
	\item[(ii)]
	$\mathscr{P} = 
	\{\text{self-adjoint rank-$d$ matrices $G \in \mathbb{F}_{q^2}^{n\times n}$ satisfying $G^2 = cG$} \}$,
	\item[(iii)]
	$\mathscr{S} = 
	\{ \text{nondegenerate $d$-dimensional subspaces of $\mathbb{F}_{q^2}^n$} \}$.
	\end{itemize}
Specifically, the functions $f \colon \mathscr{F} \to \mathscr{P}$, $g \colon \mathscr{F} \to \mathscr{S}$, and $h \colon \mathscr{P} \to \mathscr{S}$ given by
\[
f([\Phi]) = \Phi^\dagger \Phi, \qquad g([\Phi]) = \operatorname{Im} \Phi^\dagger, \qquad h(G) = \operatorname{Im} G
\]
are well-defined bijections that satisfy $g = h \circ f$.
\end{proposition}

\begin{proof}
To see that $f$ is a bijection, consider the related mapping $\Phi \mapsto \Phi^\dagger \Phi$ of a $c$-tight frame $\Phi$ to its Gramian $\Phi^\dagger \Phi$.
Its fibers consist of unitary equivalence classes of $c$-tight frames, by Proposition~\ref{prop: unitary Gram}.
On the other hand, Proposition~\ref{prop: tight frame} and Theorem~\ref{thm: factor U} imply that its range is precisely $\mathscr{P}$.
Factoring out equivalence classes creates a bijection $f \colon \mathscr{F} \to \mathscr{P}$.

Next we consider $h$.
Given $G \in \mathscr{P}$, set $W = \operatorname{Im} G$ and $P = c^{-1} G$.
Then $P^2 = P = P^*$ and $\operatorname{Im} P = W$.
In particular, $W^\perp = (\operatorname{Im} P^*)^\perp = \operatorname{Ker} P$.
For any $x \in W$, the identity $P^2 = P$ implies that $Px = x$, and for any $x \in W^\perp$ we have $Px = 0$.
It follows that $W \cap W^\perp = \{0\}$.
In other words, the mapping $G \mapsto h(G)= \operatorname{Im} G$ sends $\mathscr{P}$ into $\mathscr{S}$.
Furthermore, $h \colon \mathscr{P} \to \mathscr{S}$ is injective since $\mathbb{F}_{q^2}^n = W \oplus W^\perp$ and we have determined the action of $P= c^{-1} G$ on both spaces.
Finally, if we are given $W \in \mathscr{S}$ we may choose an orthonormal basis $w_1,\dotsc,w_d$ for $W$ and define $G = c \sum_{j\in [d]} w_j w_j^*$.
Then $G \in \mathscr{P}$ has $h(G) = \operatorname{Im} G = W$.
Hence $h \colon \mathscr{P} \to \mathscr{S}$ is a bijection, and so is $g = h \circ f$.
\end{proof}

\begin{lemma}
\label{lem: discr ran}
Assume Case O (Definition~\ref{def: Case OU}). 
If $\Phi$ is a nondegenerate frame for $V$ then 
\[
\operatorname{discr}( \operatorname{Im} \Phi^\dagger) =  \det( \Phi \Phi^\dagger) \cdot \operatorname{discr}(V).
\]
\end{lemma}

\begin{proof}
We may assume $V = \mathbb{F}_q^d$ has form $\langle x, y \rangle = x^* M y$.
Representing $\Phi$ as a matrix, we have $\det( \Phi \Phi^\dagger) = \det( \Phi \Phi^*) \det (M)$.
Since $\Phi$ is nondegenerate, $\Phi \Phi^*$ is a Gram matrix for $\operatorname{Im} \Phi^* = \operatorname{Im} \Phi^\dagger$.
Consequently, $\det(\Phi \Phi^\dagger)\, \mathbb{F}^{\times 2} = \operatorname{discr}( \operatorname{Im} \Phi^\dagger) \cdot \operatorname{discr} (V)$.
\end{proof}

\begin{proposition}
\label{prop: tight frames subspaces O}
Assume Case O (Definition~\ref{def: Case OU}).
Choose $c \in \mathbb{F}_q^\times$ and $n \geq d = \dim V$.
Then all of the following sets have the same cardinality:
	\begin{itemize}
	\item[(i)]
	$\mathscr{F}$, the set of all $\operatorname{O}(V)$-equivalence classes of $c$-tight frames for $V$ with $n$ vectors,

	\item[(ii)]
	$\mathscr{P}$, the set of all symmetric rank-$d$ matrices $G \in \mathbb{F}_{q}^{n\times n}$ satisfying $G^2 = cG$, such that a basic submatrix  $G_b$ satisfies $(\det G_b) \mathbb{F}_q^{\times 2} = \operatorname{discr}(V)$,

	\item[(iii)]
	$\mathscr{S}$, the set of all nondegenerate $d$-dimensional subspaces $W \leq \mathbb{F}_{q}^n$ with $\operatorname{discr}(W) = c^d \cdot \operatorname{discr}(V)$.
	\end{itemize}
Specifically, the functions $f \colon \mathscr{F} \to \mathscr{P}$, $g \colon \mathscr{F} \to \mathscr{S}$, and $h \colon \mathscr{P} \to \mathscr{S}$ given by
\[
f([\Phi]) = \Phi^\dagger \Phi, \qquad g([\Phi]) = \operatorname{Im} \Phi^\dagger, \qquad h(G) = \operatorname{Im} G
\]
are well-defined bijections that satisfy $g = h \circ f$.
\end{proposition}

As a consequence of Proposition~\ref{prop: tight frames subspaces O}, when $d$ is odd every nondegenerate subspace $W \leq \mathbb{F}_{q}^n$ corresponds with a tight frame for $V$, but the frame constant $c$ must satisfy $c\, \mathbb{F}_q^{\times 2} = \operatorname{discr}(W) / \operatorname{discr}(V)$.
On the other hand, when $d$ is even, a nondegenerate subspace $W \leq \mathbb{F}_{q}^n$ yields a tight frame for $V$ if and only if $\operatorname{discr}(W) = \operatorname{discr}(V)$. 
Here the frame constant may be arbitrary.

\begin{proof}
As in the proof of Proposition~\ref{prop: tight frames subspaces U}, the mapping $\Phi \mapsto \Phi^\dagger \Phi$ induces a bijection $\mathscr{F} \to \mathscr{P}$.
Next we consider $h$.
Given $G \in \mathscr{P}$, Theorem~\ref{thm: factor U} provides a $c$-tight frame $\Phi$ for $V$ such that $G = \Phi^\dagger \Phi$.
Then Lemma~\ref{lem: discr ran} implies that $\operatorname{Im} G = \operatorname{Im} \Phi^\dagger$ has discriminant $c^d \cdot \operatorname{discr} V$.
In other words, the mapping $G \mapsto h(G) = \operatorname{Im} G$ sends $\mathscr{P}$ into $\mathscr{S}$.
The same argument of Proposition~\ref{prop: tight frames subspaces U} shows that $h \colon \mathscr{P} \to \mathscr{S}$ is injective, and it remains only to prove it is sujective.

Choose any subspace $W \leq \mathbb{F}_q^n$ in $\mathscr{S}$.
To show that $W = \operatorname{Im} G$ for some $G \in \mathscr{P}$, it suffices to construct a $c$-tight frame $\Phi$ for $V$ having $\operatorname{Im} \Phi^\dagger = \operatorname{Im} \Phi^\dagger \Phi = W$.
Fix a basis $v_1,\dotsc,v_d$ for $V$, and let $M = \begin{bmatrix} \langle v_i, v_j \rangle \end{bmatrix} \in \mathbb{F}_q^{d\times d}$ be its Gram matrix.
Since
\[
\det(cM) \mathbb{F}_q^{\times 2} = c^d \cdot \det(M) \mathbb{F}_q^{\times 2} = c^d \cdot \operatorname{discr}(V) = \operatorname{discr}(W),
\]
there is a basis $w_1,\dotsc,w_d$ for $W$ satisfying $(w_i, w_j) = c \langle v_i, v_j \rangle$ for every $i,j \in [d]$.
Let $A \colon V \to \mathbb{F}_q^{n\times n}$ be the unique linear operator with $Av_i = w_i$ for every $i \in [d]$.
It is clearly injective, so $\Phi := A^\dagger \colon \mathbb{F}_q^n \to V$ is the synthesis operator of a frame for $V$.
Furthermore, we have arranged so that $( \Phi^\dagger u, \Phi^\dagger v ) = c \langle u, v \rangle$ for every $u,v \in V$, so $\Phi$ is a $c$-tight frame.
The proof is complete since $\operatorname{Im} \Phi^\dagger = W$ and $g = h \circ f$.
\end{proof}

\begin{remark}
In either Case U or Case O (Definition~\ref{def: Case OU}) we can find the number of $\operatorname{I}(V)$-equivalence classes of nondegenerate tight frames by counting subspaces of $\mathbb{F}^n$.
This can be done with an orbit-stabilizer argument for the action of an isometry group, as described on page~148 of \cite{Ar57}.
We omit details but report the results.

In Case U, choose $n \geq d= \dim V$ and $c \in \mathbb{F}_q^\times$.
Then the number of equivalence classes of $c$-tight frames for $V$ with $n$ vectors equals
\[
\frac{ | \mathrm{U}(n,q) | }{ | \mathrm{U}(d,q) | \cdot | \mathrm{U}(n-d,q) |},
\]
where $|\mathrm{U}(n,q)|$ can be found in Theorem~11.28 of~\cite{Gr02}.

A similar formula applies in Case O, but now discriminants play a role along with the frame constant $c \in \mathbb{F}_q^\times$.
For $n \geq d = \dim V$ the number of equivalence classes of $c$-tight frames for $V$ with $n$ vectors equals
\[
\frac{ |\mathrm{O}(\mathbb{F}_q^n)| }{ |\mathrm{O}(W)| \cdot | \mathrm{O}(U) | },
\]
where $\mathbb{F}_q^n$ refers to the real model and $W$ and $U$ are quadratic spaces over $\mathbb{F}_q$ such that $\dim W = d$, $\dim U = n-d$, $\operatorname{discr} W = c^d \cdot \operatorname{discr} V$, and $\operatorname{discr} U = c^d\, \mathbb{F}_q^{\times 2}$.
The orders of the orthogonal groups are given in Theorem~9.11 of~\cite{Gr02}.
\end{remark}

\begin{remark}
In contrast with nondegenerate tight frames, $I(V)$-equivalence classes of $0$-tight frames may not be identified with the range of a corresponding analysis operator.
For example, suppose $\Phi \in \mathbb{F}_{q^2}^{d\times n}$ is a $0$-tight frame under the complex model, and choose any invertible operator $A \in \operatorname{GL}(d,q^2)$.
Then $\Psi := A \Phi \in \mathbb{F}_{q^2}^{d \times n}$ is a frame, and $\operatorname{Im} \Psi^\dagger = \operatorname{Im} \Phi^\dagger$.
By Corollary~\ref{cor: iso bound}, $\Psi$ is a $0$-tight frame.
However, $A$ may be chosen so that $\Psi^\dagger \Psi = \Phi^\dagger (A^\dagger A) \Phi \neq \Phi^\dagger \Phi$, in which case $\Psi$ and $\Phi$ are unitarily inequivalent by Proposition~\ref{prop: unitary Gram}.
\end{remark}

\subsection{Naimark complements}

In real and complex frame theory, for every tight frame $\Phi \in \mathbb{C}^{d \times n}$ there exists a \emph{Naimark complement} $\Psi \in \mathbb{C}^{(n-d) \times n}$ for which the stacked array $A = \begin{bmatrix} \Phi \\ \Psi \end{bmatrix} \in \mathbb{C}^{n\times n}$ is a scalar multiple of a unitary.
For any choice of Naimark complement, the identity $cI = AA^* = \begin{bmatrix} \Phi \Phi^* & \Phi \Psi^* \\ \Psi \Phi^* & \Psi \Psi^* \end{bmatrix}$ implies that $\Psi$ is a tight frame whose analysis operator maps into the kernel of the synthesis operator for $\Phi$.
Meanwhile, the formula $cI = A^* A = \Phi^* \Phi + \Psi^* \Psi$ shows that the Gram matrices of $\Phi$ and $\Psi$ have opposite entries off the diagonal.
If $\Phi$ is an equiangular tight frame (as defined in the introduction), then so is $\Psi$.
As such, the sizes of real and complex equiangular tight frames occur in Naimark complementary pairs $(d \times n, (n-d) \times n)$.
We now give analogues of these results for finite fields.

\begin{proposition}
\label{prop: Naimark U}
Assume Case~U (Definition~\ref{def: Case OU}).
Let $\Phi$ be a $c$-tight frame ($c \neq 0$) for $V$ with Gram matrix $G \in \mathbb{F}_{q^2}^{n\times n}$.
Then $c I - G$ is the Gram matrix of a $c$-tight frame $\Psi$ for a unitary space $W$ of dimension $n-d$, such that $\operatorname{Im} \Psi^\dagger = (\operatorname{Im} \Phi^\dagger)^\perp$.
\end{proposition}

\begin{proof}
Since $G = \Phi^\dagger \Phi$ satisfies $G^2 = cG$, the kernel of $H := cI -G$ coincides with the range of $G$.
As $H=H^*$ it follows that $\operatorname{Im} H = ( \operatorname{Im} G)^\perp$, and $\operatorname{rank} H = n-d$.
Apply Theorem~\ref{thm: factor U} to write $H = \Psi^\dagger \Psi$, where $\Psi$ is a $d\times n$ frame in a unitary geometry.
Then $\Psi$ is a $c$-tight frame since $H^2 = cH$, and $\operatorname{Im} \Psi^\dagger = \operatorname{Im} H = (\operatorname{Im} \Phi^\dagger)^\perp$.
\end{proof}

\begin{proposition}
\label{prop: Naimark O}
Assume Case O (Definition~\ref{def: Case OU}).
Let $\Phi$ be a $c$-tight frame ($c\neq 0$) for $V$ with Gram matrix $G \in \mathbb{F}_q^{n\times n}$, and choose $a \in \mathbb{F}_q^\times$.
Then $a(cI - G)$ is the Gram matrix of an $ac$-tight frame $\Psi$ for a quadratic space $W$ of dimension $n-d$ such that $\operatorname{discr}(W) = a^{n-d} c^n \cdot \operatorname{discr}(V)$.
Here $\operatorname{Im} \Psi^\dagger = (\operatorname{Im} \Phi^\dagger)^\perp$.
\end{proposition}

\begin{proof}
The same argument of Proposition~\ref{prop: Naimark U} applies, and we need only compute the discriminant of $W$.
Put $H = a(cI-G)$, and let $W$ be a quadratic space of dimension $n-d$ admitting a frame $\Psi$ with Gram matrix $H$.
As $\Phi$ is nondegenerate, $\mathbb{F}_q^n = \operatorname{Im} G \oplus \operatorname{Im} H$.
Therefore $\operatorname{discr}(\operatorname{Im} G)  \cdot \operatorname{discr}(\operatorname{Im} H) = \operatorname{discr}(\mathbb{F}_q^n)$ is trivial, i.e.,~$\operatorname{discr}(\operatorname{Im} G)  = \operatorname{discr}(\operatorname{Im} H)$.
The latter are related to $\operatorname{discr}(V)$ and $\operatorname{discr}(W)$ as in Proposition~\ref{prop: tight frames subspaces O}, and, in particular, $ c^d \cdot \operatorname{discr}(V) = (ac)^{n-d} \cdot \operatorname{discr}(W)$.
\end{proof}

In Proposition~\ref{prop: Naimark O}, the choice of $a \in \mathbb{F}_q^\times$ makes a difference for $\operatorname{discr}(W)$ if and only if $n-d$ is odd.
For both Proposition~\ref{prop: Naimark U} and Proposition~\ref{prop: Naimark O} it is essential that $c \neq 0$, since a $0$-tight frame of size $d \times n$ exists only if $n \geq 2d$.

\subsection{Equal norm tight frames}

\begin{definition}
Let $\Phi = \{ \varphi_j \}_{j\in [n]}$ be a $c$-tight frame for $V$.
If there is a constant $a\in \mathbb{F}$ (possibly zero) such that $\langle \varphi_j, \varphi_j \rangle = a$ for every $j \in [n]$, then we call $\Phi$ an $(a,c)$-\textbf{equal norm tight frame}, or $(a,c)$-NTF.
\end{definition}

NTFs are generalizations of unit norm tight frames that allow arbitrary norms.
If there exists nonzero $\alpha \in \mathbb{F}$ such that $\alpha \alpha^\sigma = a$, then we may rescale an $(a,c)$-NTF to obtain a $(1,c/a)$-NTF, with unit norm.
However this is not always possible.
Some authors use the abbreviation ENTF instead of NTF.
We eschew this terminology in order to avoid confusion with the \emph{stronger} notion of equiangular tight frame (ETF).

If $\Phi$ is an $(a,c)$-NTF with $n$ vectors, then the traces of $\Phi^\dagger \Phi$ and $\Phi \Phi^\dagger$ give
\begin{equation}
\label{eq:nadc}
na = dc,
\end{equation}
where $d = \dim V$ as usual.
For example, an $(a,0)$-NTF exists only if $a=0$ or $\operatorname{char} \mathbb{F} \mid n$.

\begin{example}
\label{ex:DFT}
There are finite field versions of harmonic frames, which provide a large supply of NTFs in Case U (Definition~\ref{def: Case OU}).
If $m \mid q+1$ then a primitive $m$-th root of unity $\omega \in \mathbb{F}_{q^2}^\times$ satisfies $\omega^{q+1} = 1$, and we may create the matrix $\mathcal{F} = \begin{bmatrix} \omega^{ij} \end{bmatrix}_{ij} \in \mathbb{F}_{q^2}^{m\times m}$.
It is a \textbf{Hadamard matrix} of order $m$ since $\mathcal{F}^* \mathcal{F} = m I$ and every entry of $\mathcal{F}$ is unimodular.
By taking tensor powers, we may create a Hadamard matrix $H$ of order $n$ whenever every prime factor of $n$ divides $q+1$.
Then we may select any $d \leq n$ rows of $H$ to produce a $(d,n)$-NTF $\Phi \in \mathbb{F}_{q^2}^{d\times n}$.
\end{example}

\begin{remark}
Assume $q$ is odd in Case U (Definition~\ref{def: Case OU}).
By modifying the ``spectral tetris'' construction of~\cite{CFMWZ11}, one may create an $(a,c)$-NTF of size $d\times n$ whenever~\eqref{eq:nadc} holds and either $c \neq 0$ or $a=c=0$.
We omit details, and leave open the general problem of characterizing NTF existence.
\end{remark}

\begin{example}
\label{ex:max NTF}
In the complex model, choose any $d > 1$ and let $\Phi$ consist of one unit vector from each nonisotropic 1-dimensional subspace of $\mathbb{F}_{q^2}^d$.
Then $\Phi \in \mathbb{F}_{q^2}^{d\times n}$ is a $(1,0)$-NTF with $n = q^{d-1}\,\frac{q^d + (-1)^{d+1}}{q+1}$ vectors, as we now explain.

Let $\omega \in \mathbb{F}_{q^2}^\times$ be a generator for the subgroup $\mathbb{T}_q \leq \mathbb{F}_{q^2}^\times$ of unimodular scalars, and consider the unit sphere of $\mathbb{F}_{q^2}^d$ expressed as columns of the matrix
\[
\Psi = \left[ \begin{array}{cccc} \Phi & \omega \Phi & \cdots & \omega^q \Phi \end{array} \right] \in \mathbb{F}_{q^2}^{d \times n(q+1)}.
\]
By an inductive argument $\Psi$ has $q^{d-1}[ q^d + (-1)^{d+1}] = n(q+1)$ columns, which gives the formula for $n$.
Furthermore, the columns of $\Psi$ (hence also of $\Phi$) span $\mathbb{F}_{q^2}^d$ since the former contain the standard basis, and $\Psi \Psi^* = (q+1) \Phi \Phi^* = \Phi \Phi^*$.
Therefore it suffices to show $\Psi \Psi^* = 0$.

Any $U \in \mathrm{U}(d,q)$ permutes the unit sphere, so it commutes with $\Psi \Psi^*$.
Taking permutation matrices for $U$ we see that $\Psi \Psi^*$ has the form $cI + bJ$ for $c,b\in \mathbb{F}_q$.
Then $bJ$ also commutes with every $U \in \mathrm{U}(d,q)$.
By an application of the Witt Extension Theorem (Theorem~10.12 of~\cite{Gr02}), it follows that $b=0$ and $\Psi \Psi^* = cI$.

To get $c=0$, it suffices to show the last row $v \in \mathbb{F}_{q^2}^n$ of $\Psi$ satisfies $(v,v)=0$.
The columns $x = \left[ \begin{array}{c} y \\ \beta \end{array} \right]$ of $\Psi$ are in bijection with pairs $(y,\beta) \in \mathbb{F}_{q^2}^{d-1} \times \mathbb{F}_{q^2}$ such that $(y,y) = 1 - \beta^{q+1}$.
As such $(v,v) = \sum_{b\in \mathbb{F}_q^\times} n_b b$, where $n_b$ is the number of $y \in \mathbb{F}_{q^2}^{d-1}$ with $(y,y) = 1-b$.
For $b \neq 1$ this is the size of the unit sphere in $\mathbb{F}_{q^2}^{d-1}$, or $n_b = q^{d-2}[ q^{d-1} + (-1)^d]$.
For $b=1$ subtraction gives $n_1 = q^{2d-3} + (-1)^{d-1}( q^{d-1} - q^{d-2})$.
If $d > 2$ then every $n_b$ is divisible by $q$, so that $(v,v) = 0$.
On the other hand if $d=2$ then $q$ divides every $n_b - 1$, so that $(v,v) = \sum_{b\in \mathbb{F}_q^\times} b = 0$.
\end{example}

\begin{example}
In the real model, choose $a \in \mathbb{F}_q^\times$ and let $\Phi \in \mathbb{F}_q^{d \times n}$ consist of one vector from each pair $\{x,-x\}$ in $\mathbb{F}_q^d$ such that $(x,x) = a$.
Then $\Phi$ is an NTF, by an argument similar to that of Example~\ref{ex:max NTF}.
The frame constant may or may not be zero depending on $q,a,d$.
\end{example}

\section{Equiangular lines}
\label{sec: equiangular}

Next we develop the basic theory of equiangular lines over arbitrary fields.
In the real and complex case this is just the theory of equiangular lines.
Our main result (Theorem~\ref{thm: Gerzon}) is a generalization of Gerzon's bound.

Recall that $\mathbb{F}_0 \leq \mathbb{F}$ is the subfield fixed by $\sigma$.
In Case~O we have $\mathbb{F}_0 = \mathbb{F}_q = \mathbb{F}$, and in Case~U it is $\mathbb{F}_0 = \mathbb{F}_q \leq \mathbb{F}_{q^2} = \mathbb{F}$.

\begin{definition}
Given $a,b \in \mathbb{F}_0$, we say $\Phi = \{ \varphi_j \}_{j\in [n]}$ forms an \textbf{$(a,b)$-equiangular system} in $V$ if the following hold:
\begin{itemize}
\item[(i)]
$\langle \varphi_j, \varphi_j \rangle = a$ for every $j \in [n]$,
\item[(ii)]
$\langle \varphi_i, \varphi_j \rangle \langle \varphi_j, \varphi_i \rangle = b$ for every $i \neq j$ in $[n]$.
\end{itemize}
If this holds and $\varphi_j \neq 0$ for every $j \in [n]$, then $\mathscr{L} = \{ \operatorname{span} \varphi_j \}_{j\in [n]}$ forms a sequence of \textbf{equiangular lines}.
For $c \in \mathbb{F}_0$, $\Phi$ is an \textbf{$(a,b,c)$-equiangular tight frame}, or \textbf{$(a,b,c)$-ETF}, if the following hold in addition to (i)--(ii):
\begin{itemize}
\item[(iii)]
$\operatorname{span} \Phi = V$,
\item[(iv)]
$\Phi \Phi^\dagger = c I$.
\end{itemize}
In other words, an $(a,b,c)$-ETF is an $(a,b)$-equiangular system that is also a $c$-tight frame.
Equivalently, it is an $(a,c)$-NTF for which (ii) holds.
\end{definition}

In our general setting, we have the following version of Gerzon's bound~\cite{LS73}.
Our overall method of proof is the usual one, but the abstract setting presents a few subtleties to address.

\begin{theorem}[Gerzon's bound]
\label{thm: Gerzon}
Denote $k = \dim_{\mathbb{F}_0} \mathbb{F} \in \{1, 2\}$, depending on whether or not $\sigma$ is trivial.
Suppose $a^2 \neq b$.
Then there exists an $(a,b)$-equiangular system $\Phi$ of $n$ vectors in $V$ only if $n \leq d  + \tfrac{k}{2}(d^2-d)$.
If equality holds, then $\Phi \Phi^\dagger = c I$ for some $c \in \mathbb{F}_0$, where $c = 0$ if $a = 0$.
If equality holds in Case O or Case U (Definition~\ref{def: Case OU}), then $\Phi$ is an $(a,b,c)$-ETF.
\end{theorem}

The case $a^2 = b$ is exceptional, and in the real or complex case it may only describe vectors chosen repeatedly from a single line, as a consequence of the condition for equality in Cauchy--Schwarz.
Stranger behavior can occur over finite fields, as demonstrated by Example~\ref{ex: equiangular isotropic} further below.

\begin{proof}
Throughout the proof we work in the $\mathbb{F}_0$-space $\mathscr{S}$ of linear operators $A \colon V \to V$ satisfying $A^\dagger = A$, and we equip $\mathscr{S}$ with the (possibly degenerate) symmetric $\mathbb{F}_0$-bilinear form $\langle A, B \rangle_F = \operatorname{tr}(AB)$.
For any choice of basis $e_1,\dotsc,e_d$ of $V$, the mapping $T \colon \mathscr{S} \to \mathbb{F}^{d\times d}$ given by $T(A) = \begin{bmatrix} \langle e_i, A e_j \rangle \end{bmatrix}$ is easily seen to be an $\mathbb{F}_0$-linear isomorphism of $\mathscr{S}$ onto the space of self-adjoint $d\times d$ matrices.
By counting entries on and above the diagonal, we deduce that $\dim_{\mathbb{F}_0} \mathscr{S} = d + \tfrac{k}{2}(d^2-d)$.
This gives an upper bound on the size of a linearly independent set in $\mathscr{S}$, which we will use to prove the theorem.

Let $\Phi = \{ \varphi_j \}_{j\in [n]}$ be an $(a,b)$-equiangular system in $V$.
For each $j\in [n]$ define the outer product $A_j = \varphi_j \varphi_j^\dagger \in \mathscr{S}$ by $A_j v = \langle \varphi_j, v \rangle \varphi_j$, and let $\mathscr{A} = \{ A_j \}_{j\in [n]}$.
We start by identifying linear dependencies in $\mathscr{A}$ from its Gram matrix.
For any $i,j \in [n]$,
\[
\langle A_i, A_j \rangle_F 
= \operatorname{tr}( \varphi_i^\dagger \varphi_j \varphi_j^\dagger \varphi_i  ) = \langle \varphi_i, \varphi_j \rangle \langle \varphi_j, \varphi_i \rangle
= \begin{cases} a^2, & \text{if } i = j; \\ b, & \text{if }i \neq j. \end{cases}
\]
Therefore $\mathscr{A}$ has Gram matrix $G: = \begin{bmatrix} \langle A_i, A_j \rangle_F \end{bmatrix} = b J_n + (a^2-b) I_n \in \mathbb{F}_0^{n \times n}$, where $J_n$ is the all-ones matrix.
Since $\langle \cdot, \cdot \rangle_F$ may be degenerate, we cannot always factor $G = \mathscr{A}^\dagger \mathscr{A}$ (in particular, $\mathscr{A}^\dagger$ may not be well defined); however it is still true that $\operatorname{Ker} \mathscr{A} \leq \operatorname{Ker} G$.
Furthermore, since $a^2 - b \neq 0$ we have $\operatorname{Ker} G \leq \operatorname{Im} J_n = \operatorname{span} \{ \mathbf{1}_n \}$, where $\mathbf{1}_n \in \mathbb{F}_0^n$ is the all-ones vector.
Equality holds only if $nb + (a^2-b) = 0$.
Therefore, $\operatorname{Ker} \mathscr{A} = \{0\}$ when $a^2 \neq -(n-1)b$, and $\operatorname{Ker} \mathscr{A} \leq \operatorname{span} \{ \mathbf{1}_n \}$ generally.

First assume $a \neq 0$.
Since $\mathscr{A}$ maps $\mathbb{F}_0^n$ into $\mathscr{S}$, and since $\operatorname{Ker} \mathscr{A} \leq \operatorname{span} \{ \mathbf{1}_n \}$, rank-nullity gives the bound
\[
\operatorname{dim}_{\mathbb{F}_0} \mathscr{S} \geq \dim_{\mathbb{F}_0} \operatorname{Im} \mathscr{A} = n - \operatorname{dim}_{\mathbb{F}_0} \operatorname{Ker} \mathscr{A} \geq n-1,
\]
that is, $n \leq d + \tfrac{k}{2}(d^2-d) + 1$.
Equality holds only if $a^2 = -(n-1)b$ and $\operatorname{Ker} \mathscr{A} = \operatorname{span} \{ \mathbf{1}_n \}$, that is, $\sum_{j\in [n]} \varphi_j \varphi_j^\dagger = 0$.
We claim this cannot happen.
Otherwise, the relation $\operatorname{tr}(\Phi^\dagger \Phi) = \operatorname{tr}(\Phi \Phi^\dagger)$ says that $na = 0$, while $\operatorname{char} \mathbb{F}$ cannot divide $n$ since $b \neq a^2 = -(n-1)b$.
Therefore $a = 0$, contrary to assumption.
It follows that $n \leq d + \tfrac{k}{2}(d^2-d)$.

Now assume $a = 0$.
Then $\operatorname{tr}(A_j) = \operatorname{tr}(\varphi_j^\dagger \varphi_j) = a = 0$ for each $j$, so that $\mathscr{A}$ lies in the subspace
\[
\mathscr{S}_0 = \{ A \in \mathscr{S} : \operatorname{tr}(A) = 0 \} \leq \mathscr{S}
\]
of traceless self-adjoint operators.
Notice that $\dim_{\mathbb{F}_0} \mathscr{S}_0 = \dim_{\mathbb{F}_0} \mathscr{S} - 1$ since the trace operator maps $\mathscr{S}$ linearly onto $\mathbb{F}_0$.
Proceeding as before, we find that
\[
\dim_{\mathbb{F}_0} \mathscr{S} - 1 = \dim_{\mathbb{F}_0} \mathscr{S}_0 \geq \dim_{\mathbb{F}_0} \operatorname{Im} \mathscr{A} = n - \dim_{\mathbb{F}_0} \operatorname{Ker} \mathscr{A} \geq n - 1,
\]
i.e., $n \leq \dim_{\mathbb{F}_0} \mathscr{S}$.
Equality holds only if $0 = \sum_{j\in [n]} \varphi_j \varphi_j^\dagger = \Phi \Phi^\dagger$.

Next consider the case of equality, $n = \dim_{\mathbb{F}_0} \mathscr{S}$, with $a \neq 0$.
Define $A_{n+1} = I_d$ and $\mathscr{B} = \{ A_j \}_{j\in [n+1]}$, which equals $\mathscr{A}$ appended by the identity matrix.
For any $j \leq n$ we have $\langle A_j, A_{n+1} \rangle_F = \operatorname{tr}(\varphi_i \varphi_i^\dagger I) = a$, so $\mathscr{B}$ has Gram matrix
\[
H := \begin{bmatrix} \langle A_i, A_j \rangle_F \end{bmatrix}_{i,j \in [n+1]} 
=
\left[ \begin{array}{cc}
bJ_n + (a^2-b)I_n & a \mathbf{1}_n \\
a \mathbf{1}_n^\top & d 
\end{array} \right]
\in \mathbb{F}_0^{(n+1)\times (n+1)}.
\]
Considering that $\mathscr{B}$ has $n+1 > \dim_{\mathbb{F}_0} \mathscr{S}$ vectors, we conclude it is linearly dependent.
Choose any nonzero $x = \{ x_i \}_{i\in [n+1]} \in \operatorname{Ker} \mathscr{B}$.
As above, we have $x \in \operatorname{Ker} H$.
For any choice of $i, j \leq n$, expanding matrix products in the equation $\begin{bmatrix} Hx \end{bmatrix}_i = 0 = \begin{bmatrix} Hx \end{bmatrix}_j$ shows that $a^2 x_i + b x_j = b x_i + a^2 x_j$,
or $(a^2 - b)x_i = (a^2 - b)x_j$.
Since $a^2 \neq b$ it follows that $x_i = x_j$, and $x = \left[ \begin{array}{cc} \alpha \mathbf{1}_n^\top & \beta \end{array} \right]^\top$ for some $\alpha, \beta \in \mathbb{F}_0$.
Furthermore, $\alpha \neq 0$ since $\begin{bmatrix} Hx \end{bmatrix}_1 = 0$ and $a \neq 0$.
Defining $c = -\beta/\alpha$, we conclude
\[ 0 = \alpha^{-1} \sum_{j\in [n]} x_j A_j + \alpha^{-1} x_{n+1} I_d = \sum_{j \in [n]} \varphi_j \varphi_j^\dagger - c I_d. \]
Therefore, $\Phi \Phi^\dagger = c I_d$ as desired. 

Finally, suppose that the bound is attained with $n =d + \tfrac{k}{2}(d^2 - d)$ in either Case~O or Case~U.
Then 
\begin{equation}
\label{eq: Gerz 1}
d \geq \operatorname{rank} \Phi \geq \operatorname{rank} \Phi^\dagger \Phi =: d'.
\end{equation}
By Theorem~\ref{thm: factor O} (Case O) or Theorem~\ref{thm: factor U} (Case U), there is a frame $\Psi$ of $n$ vectors in an orthogonal (Case O) or unitary (Case U) geometry on $\mathbb{F}^{d'}$, such that $\Psi^\dagger \Psi = \Phi^\dagger \Phi$.
Then $\Psi$ is an $(a,b)$-equiangular system in dimension $d'$, and considering the bound proved above we must have $d' = d$.
Therefore equality holds throughout~\eqref{eq: Gerz 1}, and $\Phi$ is a frame.
By the above it is an ETF.
\end{proof}

We now demonstrate the pathology of the case $a^2 = b$ for Gerzon's bound.

\begin{example}
\label{ex: equiangular isotropic}
For $d \geq 3$ there is no finite upper bound $f(d)$ on the size of an equiangular system in any space of dimension $d$ over any field.
We cannot even bound the number of distinct lines spanned by vectors in an equiangular system without accounting for the base field.
For example, a sequence of vectors forms a $(0,0)$-equiangular system if and only if they span a totally isotropic subspace.
When $d \geq 4$ is fixed and a prime power $q$ is allowed to vary, the number $N(d,q)$ of distinct lines in a maximal totally isotropic subspace of $\mathbb{F}_{q^2}^d$ under the complex model grows to infinity with $q$.
Hence there exist arbitrarily large $(0,0)$-equiangular systems in $d$-dimensional spaces (over various fields).
\end{example}

\begin{example}
\label{ex: equiangular isotropic 2}
More generally, let $a \in \mathbb{F}_0$ be such that there exists $\varphi_0 \in V$ with $\langle \varphi_0, \varphi_0 \rangle = a$.
Choose a totally isotropic subspace $W \leq \operatorname{span} \{ \varphi_0 \}^\perp$.
For $w \in W$ define $\varphi_w = \varphi_0 + w$.
Then $\Phi = \{ \varphi_w \}_{w \in W}$ is an $(a,a^2)$-equiangular system.
This gives very large examples.
\end{example}

\begin{example}
Not every $(a,a^2)$-equiangular system takes the form of the last example.
To see this, choose $d \geq 3$ and an odd prime power $q$, and consider the complex model on $V = \mathbb{F}_{q^2}^d$.
For any isotropic $x,y \in \mathbb{F}_{q^2}^{d-1}$ satisfying $(x,y) = -2$,
\[
\Phi = \left[ \begin{array}{ccc}
1 & 1 & 1 \\
0 & x & y
\end{array} \right]
\]
is a $(1,1)$-equiangular system that does not arise from the method of Example~\ref{ex: equiangular isotropic 2}.
\end{example}

In comparison with the real and complex cases, one might expect to find a relative bound on the size of an $(a,b)$-equiangular system in Case O or Case U that beats Gerzon when we know the values of $a,b$.
We leave this as an open problem.

\begin{problem}
\label{prob: relative bound}
In Case O and Case U (Definition~\ref{def: Case OU}), find a relative bound on the size of an $(a,b)$-equiangular system that outperforms Gerzon (Theorem~\ref{thm: Gerzon}).
\end{problem}

Despite not yet having a relative bound, we can relate the parameters of an $(a,b,c)$-ETF, as in Welch~\cite{W74}.

\begin{proposition}
\label{prop: Welch}
If $V$ admits an $(a,b,c)$-ETF of $n$ vectors, then
$a(c-a) = (n-1)b$.
When $\operatorname{char} \mathbb{F}$ fails to divide both $d$ and $n-1$, it follows that $b = \frac{ ( n - d ) }{d (n-1) }a^2$.
\end{proposition}

\begin{proof}
Let $\Phi = \{ \varphi_j \}_{j\in [n]}$ be an $(a,b,c)$-ETF in $V$, and consider the matrix $A = \Phi^\dagger \Phi - aI$, whose diagonal is zero.
Expanding $(\Phi^\dagger \Phi - aI)^2$ with the relation $(\Phi^\dagger \Phi)^2 = c \Phi^\dagger \Phi$ shows that $A^2 = (c-2a)A + a(c-a)I$, and in particular, $(A^2)_{ii} = a(c-a)$ for every $i \in [n]$.
On the other hand, when we compute the matrix product we find that $(A^2)_{ii} = \sum_{j\in [n]} A_{ij} A_{ji} = (n-1)b$ for every $i \in [n]$.
Comparing these expressions shows that $a(c-a) = (n-1)b$.
Finally, when $\operatorname{char} \mathbb{F}$ fails to divide $d$ and $n-1$, we can solve~\eqref{eq:nadc} to find $c = \tfrac{n}{d}a$, so that $b = \frac{ ( n - d ) }{d (n-1) }a^2$.
\end{proof}

As in the real and complex settings, ETFs over finite fields often come in Naimark complementary pairs.
More precisely, we have the following consequence of Proposition~\ref{prop: Naimark U} and Proposition~\ref{prop: Naimark O}.

\begin{proposition}
[Naimark complements of ETFs]
\label{prop: Naimark ETF}
~
\begin{itemize}
\item[(a)]
If there exists an $(a,b,c)$-ETF of $n$ vectors in a unitary geometry on $\mathbb{F}_{q^2}^d$ and $c \neq 0$, then there exists a $({c-a},b,c)$-ETF of $n$ vectors in a unitary geometry on $\mathbb{F}_{q^2}^{n-d}$.
\item[(b)]
If there exists an $(a,b,c)$-ETF of $n$ vectors in an orthogonal geometry on $\mathbb{F}_q^d$ and $c \neq 0$, then there exists a $({c-a},b,c)$-ETF of $n$ vectors in an orthogonal geometry on $\mathbb{F}_q^{n-d}$.
\end{itemize}
\end{proposition}

We emphasize that $c \neq 0$ in Proposition~\ref{prop: Naimark ETF}, and that the orthogonal geometries in Proposition~\ref{prop: Naimark ETF}(b) may have different discriminants.

\part{ETFs in unitary geometry}
\label{part:2}

For the remainder of the paper, we focus on ETFs in finite unitary geometries.
ETFs in finite orthogonal geometries are the subject of the companion paper~\cite{FFF2}.

\section{First examples}
\label{sec: first ex}

In this section we demonstrate some constructions of ETFs in unitary geometries, focusing especially on those derived from modular difference sets.
We have not investigated finite field analogs of other sources of complex ETFs, such as Steiner systems~\cite{FMT12},  hyperovals~\cite{FJM16}, graph coverings~\cite{CGSZ16,FJMPW17,IM:DTI}, the Tremain construction~\cite{FJMP18}, association schemes~\cite{DGS77,IJM:NAG}, or Gelfand pairs~\cite{IJM:FGA}.
We leave these topics for future research.

\begin{example}
Every ETF in a finite orthogonal geometry produces one in a finite unitary geometry, as we now explain.
If $\Phi$ is an $(a,b,c)$-ETF of $n$ vectors in an orthogonal geometry on $\mathbb{F}_q^d$, then its Gram matrix $G$ may be viewed as an element of $\mathbb{F}_{q^2}^{n\times n}$ with $\operatorname{rank}_{\mathbb{F}_{q^2}} G = \operatorname{rank}_{\mathbb{F}_q} G = d$.
By Theorem~\ref{thm: factor U}, there is a frame $\Psi$ of $n$ vectors in a unitary geometry on $\mathbb{F}_{q^2}^d$ having $G$ as its Gram matrix.
Considering the entries of $G$ and the fact that $G^2 = cG$, we conclude that $\Psi$ is an $(a,b,c)$-ETF in a unitary geometry.
This creates a large number of examples, which are explored more fully in~\cite{FFF2}.
\end{example}

Next we show three ETFs in unitary geometries having unusual sizes.
The authors discovered Examples~\ref{ex:5x16},~\ref{ex:6x28},~\ref{ex:6x27} while searching for large $(0,1)$-equiangular systems, using the clique method described by Lemma~6.2 of~\cite{FFF2}.
This amounts to a computationally efficient way to find a maximum clique in the graph whose vertices are isotropic vectors in $\mathbb{F}_{q^2}^d$ under the complex model, with vertices $u$ and $v$ adjacent precisely when $(u,v)^{q+1}=1$.

\begin{example}
\label{ex:5x16}
Take $q=3$ and $\zeta \in \mathbb{F}_{3^2}^\times$ as a primitive element.
Then the following is a $(0,1,0)$-ETF of size $5 \times 16$ in the complex model on $\mathbb{F}_{3^2}^5$:
\[
\Phi =
\left[ \begin{array}{cccccccccccccccc}
1 & 2 & 0 & 0 & 0 & 0 & 0 & 0 & \zeta & \zeta & \zeta & \zeta & \zeta^3 & \zeta^3 & \zeta^3 & \zeta^3 \\
\zeta & \zeta & \zeta^5 & \zeta^5 & \zeta^5 & \zeta^5 & \zeta^5 & \zeta^5 & 1 & 1 & 1 & 1 & \zeta^6 & \zeta^6 & \zeta^6 & \zeta^6 \\
0 & 0 & 0 & 0 & 0 & 0 & 1 & 2 & \zeta & \zeta & \zeta^5 & \zeta^5 & \zeta^3 & \zeta^3 & \zeta^7 & \zeta^7 \\
0 & 0 & 0 & 0 & 1 & 2 & 0 & 0 & \zeta & \zeta^5 & \zeta & \zeta^5 & \zeta^3 & \zeta^7 & \zeta^3 & \zeta^7 \\
0 & 0 & \zeta^2 & \zeta^6 & 0 & 0 & 0 & 0 & \zeta^7 & \zeta^3 & \zeta^3 & \zeta^7 & \zeta & \zeta^5 & \zeta^5 & \zeta
\end{array} \right].
\]
\end{example}

\begin{example}
\label{ex:6x28}
Take $q=3$ and $\zeta \in \mathbb{F}_{3^2}^\times$ as a primitive element.
Then the following is a $(0,1,0)$-ETF of size $6 \times 28$ in the complex model on $\mathbb{F}_{3^2}^{28}$:
\begin{center}
\resizebox{\textwidth}{!}{
$
\displaystyle
\Phi =
\left[ \begin{array}{cccccccccccccccccccccccccccc}
1 & 2 & 0 & 0 & 0 & 0 & 0 & 0 & 0 & 0 & 1 & 1 & \zeta & \zeta & \zeta & \zeta & \zeta & \zeta & \zeta & \zeta & \zeta^3 & \zeta^3 & \zeta^3 & \zeta^3 & \zeta^3 & \zeta^3 & \zeta^3 & \zeta^3 \\
\zeta & \zeta & \zeta^5 & \zeta^5 & \zeta^5 & \zeta^5 & \zeta^5 & \zeta^5 & \zeta^5 & \zeta^5 & 0 & 0 & 1 & 1 & 1 & 1 & 1 & 1 & 1 & 1 & \zeta^6 & \zeta^6 & \zeta^6 & \zeta^6 & \zeta^6 & \zeta^6 & \zeta^6 & \zeta^6 \\
0 & 0 & 0 & 0 & 1 & \zeta^5 & \zeta^6 & \zeta^6 & \zeta^6 & \zeta^7 & 1 & 2 & 0 & 2 & 2 & \zeta & \zeta & \zeta^2 & \zeta^6 & \zeta^7 & 0 & 1 & 2 & \zeta^2 & \zeta^2 & \zeta^5 & \zeta^7 & \zeta^7 \\
0 & 0 & \zeta & \zeta^3 & \zeta^5 & \zeta^2 & 0 & \zeta^5 & \zeta^6 & \zeta^6 & \zeta^7 & \zeta^3 & 0 & \zeta^2 & \zeta^5 & 0 & \zeta^2 & \zeta^7 & \zeta^2 & 1 & \zeta^5 & \zeta & \zeta & \zeta^2 & \zeta^7 & 1 & \zeta & \zeta^6 \\
0 & 0 & 0 & 1 & \zeta & 2 & \zeta^3 & \zeta^3 & 2 & \zeta^7 & \zeta & \zeta^5 & 1 & \zeta^5 & \zeta & \zeta^7 & \zeta & \zeta & \zeta^2 & \zeta^6 & \zeta & 0 & \zeta^3 & 0 & \zeta^6 & \zeta^6 & 2 & \zeta^6 \\
0 & 0 & \zeta^3 & 1 & \zeta & 0 & 2 & \zeta & 1 & \zeta^3 & 0 & 0 & \zeta^3 & \zeta^3 & \zeta^2 & \zeta^5 & \zeta^6 & \zeta^2 & 0 & \zeta^5 & \zeta^7 & 0 & \zeta^6 & \zeta^2 & \zeta^7 & \zeta & \zeta^6 & \zeta
\end{array} \right].
$
}
\end{center}
\end{example}

\begin{example}
\label{ex:6x27}
Take $q=2$ and $\zeta \in \mathbb{F}_{2^2}^\times$ as a primitive element.
Then the following is a $(0,1,1)$-ETF of size $6 \times 27$ in the complex model on $\mathbb{F}_{2^2}^{27}$:
\begin{center}
\resizebox{\textwidth}{!}{
$
\displaystyle
\Phi =
\left[ \begin{array}{ccccccccccccccccccccccccccc}
1 & \zeta & 0 & 0 & 0 & 0 & 0 & 0 & 1 & 1 & 1 & 1 & 1 & 1 & 1 & 1 & 1 & 1 & 1 & 1 & \zeta & \zeta & \zeta & \zeta & \zeta & \zeta & \zeta \\
\zeta & \zeta & \zeta & \zeta & \zeta & \zeta & \zeta & \zeta & 0 & 0 & 0 & 0 & 0 & 0 & 0 & 0 & 0 & 0 & 0 & 0 & 1 & 1 & 1 & 1 & 1 & 1 & 1 \\
0 & 0 & 0 & 0 & 0 & 0 & 0 & 0 & 0 & 0 & 0 & 0 & 0 & 0 & 1 & 1 & \zeta & \zeta & \zeta^2 & \zeta^2 & 0 & 0 & 0 & 0 & 1 & \zeta & \zeta^2 \\
0 & 0 & 1 & 1 & 1 & \zeta & \zeta^2 & \zeta^2 & 1 & 1 & \zeta & \zeta^2 & \zeta^2 & \zeta^2 & \zeta^2 & \zeta^2 & \zeta^2 & \zeta^2 & \zeta^2 & \zeta^2 & 0 & 0 & \zeta & \zeta & \zeta & \zeta & \zeta \\
0 & 0 & 1 & \zeta & \zeta & 0 & 1 & \zeta^2 & \zeta & \zeta^2 & 0 & 1 & 1 & \zeta & 0 & \zeta & 0 & \zeta & 0 & \zeta & 1 & \zeta & 0 & \zeta^2 & \zeta & \zeta & \zeta \\
0 & 0 & \zeta & \zeta & 1 & 0 & \zeta^2 & 1 & \zeta^2 & \zeta & 0 & \zeta & 1 & 1 & \zeta & 0 & \zeta & 0 & \zeta & 0 & \zeta & 1 & \zeta^2 & 0 & \zeta & \zeta & \zeta
\end{array} \right].
$
}
\end{center}
\end{example}

\smallskip

\subsection{ETFs from modular difference sets}

Next we show how ETFs in the complex model can be constructed from the following generalization of difference sets.

\begin{definition}[\cite{Ma74}]
\label{def: mds} 
Let \(k,n\in\mathbb{N}\). A set \(D\subseteq \mathbb{Z}/n\mathbb{Z}\) is called a \textbf{\(k\)-modular difference set} if the function \(c\colon \mathbb{Z}/n\mathbb{Z} \to \mathbb{Z}/k\mathbb{Z}\) given by
\[c(g) = |\{(a,b)\in D^{2} : a-b=g\}| \mod k\]
is constant on $(\mathbb{Z}/n\mathbb{Z})\setminus \{0\}$.
\end{definition}

In order to convert modular difference sets to ETFs, we now define the discrete Fourier transform matrix over a finite field. 
For simplicity we will only consider the Fourier transform over finite cyclic groups, rather than the more general case of finite abelian groups.

\begin{definition}
\label{def: harmframe}
Let \(q\) be a power of the prime \(p\), and let \(\alpha\) be a generator of the multiplicative group of \(\mathbb{F}_{q^2}\). 
Given \(n\in\mathbb{N}\) such that \(n\mid q+1\), set \(\omega = \alpha^{(q^2-1)/n}\). 
Define the \(n\times n\) discrete Fourier transform (DFT) matrix over \(\mathbb{F}_{q^{2}}\) by
\[\mathcal{F} = \begin{bmatrix} \omega^{ij} \end{bmatrix}_{i,j\in\mathbb{Z}/n\mathbb{Z}}.\]
For \(D\subset \mathbb{Z}/n\mathbb{Z}\) we define the \(|D|\times n\) submatrix
\[\mathcal{F}_{D} = \begin{bmatrix} \omega^{ij} \end{bmatrix}_{i\in D,j\in\mathbb{Z}/n\mathbb{Z}}.\]
\end{definition}

Using the notation of Definition \ref{def: harmframe}, the next theorem shows that the matrix \(\mathcal{F}_{D}\) is an ETF if and only if \(D\) is a \(p\)-modular difference set. 
However, we wish to emphasize that the assumption \(n\mid q+1\) is essential. 
Indeed, there are \(p\)-modular difference sets that do not give rise to ETFs over a finite field because they fail to satisfy this condition.
The following theorem generalizes a construction of complex ETFs due to Strohmer and Heath~\cite{SH03,XZG05,DF07}.

\begin{theorem}
\label{thm: harm diff}
Suppose $q$ is a prime power, and $n \in \mathbb{N}$ satisfies $n \mid q+1$.
Given $D \subseteq \mathbb{Z}/n\mathbb{Z}$, the matrix \(\mathcal{F}_{D}\) of Definition~\ref{def: harmframe} is an ETF if and only if \(D\) is a \(p\)-modular difference set.
\end{theorem}

\begin{proof} 
Let \(D = \{ a_{1}, a_{2}, \dotsc, a_{d}\}\subseteq \mathbb{Z}/n\mathbb{Z} \).
As in Definition~\ref{def: mds}, we let $c \colon \mathbb{Z}/n\mathbb{Z} \to \mathbb{F}_q$ be defined by
\[
c(g) = | \{ (a,b) \in D^2 : a - b = g \}| \mod p,
\]
and we also consider $c$ in vector form as
\(\mathbf{c} = \left[\begin{array}{ccccc} c({0}) & \cdots & c({n-1})\end{array}\right]^{\top}\in\mathbb{F}_{q^{2}}^{n}.\)
For each \(j\in \mathbb{Z}/n\mathbb{Z} \), let \(\varphi_{j}\) denote the \(j\)th column of \(\mathcal{F}_{D}\), that is, 
\[\varphi_{j} = \left[\begin{array}{ccccc} \omega^{a_{1}j} & \omega^{a_{2}j} & \omega^{a_{3}j} & \cdots & \omega^{a_{d}j}\end{array}\right]^{\top}.\]

A simple calculation shows that
\[\mathcal{F}_{D}\mathcal{F}_{D}^{\ast} = n I.\]
Since \(n\mid q+1\), we see that \(p\nmid n\) and hence \(n\neq 0\). 
This implies that \(\mathcal{F}_{D}\) is an \(n\)-tight frame which is not totally isotropic. 
Since the entries in \(\mathcal{F}_{D}\) are unimodular, we see that \((\varphi_{i},\varphi_{i}) = d = |D|\). 
Thus, for any set \(D\subseteq \mathbb{Z}/n\mathbb{Z} \), the matrix \(\mathcal{F}_{D}\) is a \((|D|,n)\)-NTF.

Next, note that
\begin{equation}
\label{eq: ds1}
\begin{split}
(\varphi_{i},\varphi_{j})(\varphi_{j},\varphi_{i}) &= \operatorname{tr}(\varphi_{i}^{\ast}\varphi_{j}\varphi_{j}^{\ast}\varphi_{i}) = \operatorname{tr}(\varphi_{i}\varphi_{i}^{\ast}\varphi_{j}\varphi_{j}^{\ast}) = \sum_{k=1}^{d}\sum_{l=1}^{d}\omega^{(j-i)(a_{k}-a_{l})}\\
 & = \sum_{g\in\mathbb{Z}/n\mathbb{Z}}c({g}) \omega^{(j-i)g} = [\mathcal{F}\mathbf{c}]_{j-i}.
\end{split}
\end{equation}

Assume \(D\) is a \(p\)-modular difference set.
That is, there is a number \(\lambda\in{\{0,\ldots,p-1\}}\) such that \(c({g}) = \lambda\ (\mathrm{mod}\ p)\) for all \(g\in(\mathbb{Z}/n\mathbb{Z})\setminus\{0\}\). 
Now, for \(i\neq j\), from \eqref{eq: ds1} we see that
\[(\varphi_{i},\varphi_{j})(\varphi_{j},\varphi_{i}) = \sum_{g\in (\mathbb{Z}/n\mathbb{Z}) \setminus\{0\}}c({g})\omega^{(j-i)g} + d = \lambda\sum_{g\in(\mathbb{Z}/n\mathbb{Z})\setminus\{0\}}\omega^{(j-i)g} + d = -\lambda+d.\]
Therefore, \(\{\varphi_{i}\}_{i\in\mathbb{Z}/n\mathbb{Z}}\) is a \((|D|,|D|-\lambda,n)\)-ETF.

Finally, we assume \(\mathcal{F}_{D}\) is an ETF. 
From \eqref{eq: ds1} we can deduce that there is a constant \(\alpha\in\mathbb{F}_{q^{2}}\) such that
\[\mathcal{F}\mathbf{c} = \left[\begin{array}{ccccc} n^2 & \alpha & \alpha & \cdots & \alpha\end{array}\right]^{\top}.\]
Since \(\mathcal{F}^{\ast}\mathcal{F} = nI\neq 0\), we see that 
\[\mathbf{c} = n^{-1}\mathcal{F}^{\ast}\mathcal{F}\mathbf{c} = n^{-1}\left[\begin{array}{ccccc} n^2-(n-1)\alpha & n^2-\alpha & n^2-\alpha & \cdots & n^2-\alpha\end{array}\right]^{\top}.\]
Thus, \(D\) is a \(p\)-modular difference set.
\end{proof}

\begin{example}
\label{ex:harm7x14}
For $n = 14$, the reader can check that $D = \{ 0, 4, 6, 7, 8, 11, 13 \}$ is a $3$-modular difference set.
Since $n \mid 27 + 1$, this produces a $7\times 14$ ETF in a unitary geometry on $\mathbb{F}_{27^2}$.
A complex ETF of this size is known to exist, but it cannot be harmonic~\cite{FM:T}.
Indeed, there is no difference set of $d=7$ elements in a group of $n=14$ elements since $n-1$ fails to divide $d(d-1)$~\cite{JPS07}.
\end{example}

\begin{example}
[Example~5.2 of~\cite{Ma74}]
Let \(k\in\mathbb{N}\) such that \(p=3k-1\) is prime, and set \(q=p^3\). 
Let \(H\) be a subgroup of \(\mathbb{Z}/9k\mathbb{Z}\) such that \(|H|=3k\). 
We claim that the set \(D=H\cup\{1\}\) is a \(p\)-modular difference set. 
Indeed, consider the function \(c\) from Definition \ref{def: mds}. 
One can easily deduce that
\[c(g) = \begin{cases} 3k, & \text{if }g\in H\setminus\{0\};\\ 1, & \text{if }g\in (\mathbb{Z}/9k\mathbb{Z})\setminus H.\end{cases}\]
Since \(3k=1\ (\mathrm{mod}\ p)\), this shows that \(D\) is a \(p\)-modular difference set. 
Additionally, we see that \(q+1 = (3k-1)^3+1 = 0\ (\mathrm{mod}\ 9k)\), and hence \(9k\mid q+1.\) 
By the above theorem, the matrix \(\mathcal{F}_{D}\) is a \((3k+1)\times 9k\) ETF over \(\mathbb{F}_{q^{2}}\). 
Since there are infinitely many \(k\) such that \(3k-1\) is prime, this gives rise to an infinite family of ETFs. 
In particular, when \(k=2\) this construction produces a \(7\times 18\) ETF over \(\mathbb{F}_{125^2}\). 
Note that there is no known construction of a complex ETF of this size, and there cannot be a $7\times 18$ complex harmonic ETF for the same reason as in Example~\ref{ex:harm7x14}~\cite{FM:T}. 
\end{example}

\section{ETFs from translation and modulation}
\label{sec: Gabor}

In this section we introduce translation and modulation operators over finite fields, and show that they can be used to create NTFs just as in the complex setting.
By identifying an appropriate fiducial vector, we show that Gerzon's bound is attained in unitary geometries of every dimension $d = 2^{2l+1}$ over the field $\mathbb{F}_{3^2}$.

\subsection{NTFs from translation and modulation}
Fix a prime power $q$, and choose integers $d_1, \dotsc, d_m \geq 2$ that all divide $q+1$.
In this section we consider a finite field version of the Heisenberg group over $G := \prod_{k=1}^m \mathbb{Z}/d_k\mathbb{Z}$, where $|G| = \prod_{k=1}^m d_k =: d$.
This presents a finite model in which to investigate Zauner's conjecture.
We work in the unitary space $V =\mathbb{F}_{q^2}^G$ of functions $\varphi \colon G \to \mathbb{F}_{q^2}$, equipped with the form
\[
\langle \varphi, \psi \rangle = \sum_{x\in G} \varphi(x)^q \psi(x).
\]
Then $V$ has orthonormal basis $\{ \delta_x \}_{x\in G}$, where $\delta_x \in \mathbb{F}_{q^2}^d$ is the indicator function of a point, with $\delta_x(x) = 1$ and $\delta_x(y) = 0$ for $y \neq x$.

In order to define modulation, we first introduce notation for $\mathbb{T}_q$-valued characters on $G$.
For each $k \in [m]$, we fix a generator $\omega_k \in \mathbb{T}_q \leq \mathbb{F}_{q^2}^\times$ for the unique subgroup of order $d_k$.
Given $x_k, y_k \in\mathbb{Z}/d_k\mathbb{Z}$ we denote $\hat{y_k}(x_k) = \omega_k^{x_k y_k} = \hat{x_k}(y_k)$, and for $x=(x_1,\dotsc,x_m)$ and $y=(y_1,\dotsc,y_m)$ in $G$ we define $\hat{y}(x) = \prod_{k=1}^m \hat{y_k}(x_k) \in \mathbb{T}_q$.
Then
\begin{equation}
\label{eq: char rels}
\hat{y}(x) = \hat{x}(y) \quad \text{and} \quad (x+y)\hat{\,}(z) = \hat{x}(z) \hat{y}(z), \qquad \text{for }x,y,z\in G,
\end{equation}
and $y = 0$ if and only if $\hat{y}(x) = 0$ for every $x \in G$.
Furthermore, for any $x \in G$
\[
\hat{y}(x) \sum_{z\in G} \hat{y}(z) = \sum_{z\in G} \hat{y}(x + z) = \sum_{z\in G} \hat{y}(z),
\]
that is, $\bigl(\hat{y}(x) - 1\bigr) \sum_{z \in G} \hat{y}(z) = 0$.
It follows that
\begin{equation}
\label{eq: character sum}
\sum_{z\in G} \hat{y}(z) = 
\begin{cases}
d, & \text{if } y = 0; \\
0, & \text{otherwise.}
\end{cases}
\end{equation}

In this notation, each $y \in G$ determines a modulation operator $M_y \in \operatorname{U}( \mathbb{F}_{q^2}^G )$ and a translation operator $T_y \in \operatorname{U}( \mathbb{F}_{q^2}^G )$ given by 
\[
(M_y \varphi)(x) = \hat{y}(x) \varphi(x) \quad \text{and} \quad (T_y \varphi)(x) = \varphi(x-y), \qquad \text{for }\varphi \in \mathbb{F}_{q^2}^G,\, x \in G.
\]
It is straightforward to verify the usual relations
\begin{equation}
\label{eq: TM rels}
T_x T_y = T_{x+y},
\quad
M_x M_y = M_{x+y},
\quad 
M_y T_x = \hat{y}(x) T_x M_y,
\qquad
\text{for }x,y \in G.
\end{equation}

Our immediate goal is the following.

\begin{proposition}
\label{prop: Gabor tight frame}
For any nonzero $\varphi \in \mathbb{F}_{q^2}^G$, the collection $\Phi = \{ T_x M_y \varphi \}_{x,y \in G}$ is an $(a,da)$-NTF for $\mathbb{F}_{q^2}^G$, where $a = \langle \varphi, \varphi \rangle$.
\end{proposition}

The NTF $\Phi$ of Proposition~\ref{prop: Gabor tight frame} is known as a \textbf{Gabor frame}, and $\varphi$ is its \textbf{fiducial vector}.
If $\Phi$ happens to be an ETF then it is known as a \textbf{Gabor ETF}.

\begin{remark}
Since every NTF is a spanning set, Proposition~\ref{prop: Gabor tight frame} implies in particular that the Heisenberg group $\langle T_x, M_y : x,y \in G \rangle$ acts irreducibly in its natural representation on $\mathbb{F}_{q^2}^G$.
Indeed, if $W \leq \mathbb{F}_{q^2}^G$ is a nonzero invariant subspace, and if $\varphi \in W$ is nonzero, then $W \geq \operatorname{span}\{ T_x M_y \varphi : x,y \in G \} = \mathbb{F}_{q^2}^G$, and $W = \mathbb{F}_{q^2}^G$.
\end{remark}

In order to prove Proposition~\ref{prop: Gabor tight frame} we leverage a unitary geometry on the space $\operatorname{HS}(\mathbb{F}_{q^2}^G)$ of all linear operators on $\mathbb{F}_{q^2}^G$.
Given two such operators $A,B$ we define 
\[
\langle A, B \rangle_F 
= \operatorname{tr}(A^\dagger B)
= \sum_{z\in G} \langle A \delta_z, B \delta_z \rangle.
\]
If $A$ and $B$ are represented by their matrices $\begin{bmatrix} A_{xy} \end{bmatrix}$ and $\begin{bmatrix} B_{xy} \end{bmatrix}$ over the standard basis $\{ \delta_x \}_{x\in G}$, then $\langle A, B \rangle_F = \sum_{x,y \in G} A_{xy}^q B_{xy}$.
It follows easily that $\langle \cdot, \cdot \rangle_F$ is a nondegenerate form, and $\operatorname{HS}(\mathbb{F}_{q^2}^G)$ is a unitary space.

\begin{lemma}
The collection $\{ T_x M_y \}_{x,y \in G}$ is an orthogonal basis for $\operatorname{HS}(\mathbb{F}_{q^2}^G)$, and any $A \in \operatorname{HS}(\mathbb{F}_{q^2}^G)$ satisfies
\begin{equation}
\label{eq: HS expansion}
\sum_{x,y\in G} \langle T_x M_y, A \rangle_F T_x M_y = dA.
\end{equation}
\end{lemma}

\begin{proof}
For any $x,y \in G$ we compute
\[
\operatorname{tr}(T_x M_y)
= \sum_{z\in G} \langle \delta_z, T_x M_y \delta_z \rangle
= \sum_{z\in G} (T_x M_y \delta_z)(z)
= \sum_{z\in G} \hat{y}(z-x) \delta_z(z-x).
\]
This is clearly $0$ if $x \neq 0$, and otherwise it equals $\sum_{z\in G} \hat{y}(z)$.
By~\eqref{eq: character sum}
\[
\operatorname{tr}(T_x M_y)
= \begin{cases}
d, & \text{if }x = y = 0; \\
0, & \text{otherwise}.
\end{cases}
\]
For any $x,y,s,t \in G$ the relations~\eqref{eq: TM rels} and~\eqref{eq: char rels} now imply
\begin{equation}
\label{eq: TM orthog}
\langle T_x M_y, T_s M_t \rangle_F
= \operatorname{tr}( M_{-y} T_{-x} T_s M_t )
= \hat{y}(x-s) \operatorname{tr}(T_{s-x} M_{t - y})
=
d \delta_{x, s} \delta_{y, t}.
\end{equation}
Furthermore, $\langle T_x M_y, T_x M_y \rangle = d \not\equiv 0$ since $d = \prod_{k=1}^m d_k$ and each $d_k$ is coprime with $q$.
It follows easily that $\{ T_x M_y \}_{x,y\in G}$ is an orthogonal basis for its span, which must equal $\operatorname{HS}( \mathbb{F}_{q^2}^G)$ by dimension count.
Finally, for any $A \in \operatorname{HS}( \mathbb{F}_{q^2}^G)$ we can expand $A = \sum_{x,y \in G} c_{xy} T_x M_y$ with $c_{xy} \in \mathbb{F}_{q^2}$, and for any $s,t \in G$ the identity~\eqref{eq: TM orthog} produces
\[
\langle T_s M_t, A \rangle_F = \sum_{x,y \in G} c_{xy} \langle T_s M_t, T_x M_y \rangle_F = d c_{st}.
\]
This implies~\eqref{eq: HS expansion}.
\end{proof}

\begin{proof}[Proof of Proposition~\ref{prop: Gabor tight frame}]
Fix nonzero $\varphi \in \mathbb{F}_{q^2}^G$.
Given $\psi \in \mathbb{F}_{q^2}^G$ we denote $\psi \varphi^\dagger \in \operatorname{HS}(\mathbb{F}_{q^2}^G)$ for the operator $(\psi \varphi^\dagger)(\vartheta) = \langle \varphi, \vartheta \rangle \psi$, $\vartheta \in \mathbb{F}_{q^2}^G$.
Then for any choice of $A \in \operatorname{HS}( \mathbb{F}_{q^2}^G )$
\[
\langle A, \psi \varphi^\dagger \rangle_F
= \sum_{x\in G} \langle A \delta_x, \psi \varphi^\dagger \delta_x \rangle
= \sum_{x\in G} \Big\langle A \delta_x, \langle \varphi, \delta_x \rangle \psi \Big\rangle
= \Big\langle A \sum_{x\in G} \langle \delta_x, \varphi \rangle \delta_x , \psi \Big\rangle.
\]
That is, $\langle A, \psi \varphi^\dagger \rangle_F = \langle A \varphi, \psi \rangle$.
Applying this identity and~\eqref{eq: HS expansion} we find
\[
\sum_{x,y \in G} \langle T_x M_y \varphi, \psi \rangle T_x M_y \varphi
=
\sum_{x,y \in G} \langle T_x M_y, \psi \varphi^\dagger \rangle_F T_x M_y \varphi
=
d (\psi \varphi^\dagger)(\varphi)
= d a \psi.
\]
Therefore $\Phi \Phi^\dagger = da I$.
Furthermore, if $\psi \neq 0$ then $\psi \varphi^\dagger \neq 0$ since there exists $\vartheta \in \mathbb{F}_{q^2}^G$ with $\langle \varphi, \vartheta \rangle \neq 0$.
Consequently, there exist $x,y \in G$ with $0 \neq \langle T_x M_y, \psi \varphi^\dagger \rangle_F = \langle T_x M_y \varphi, \psi \rangle$.
Since this holds for every $\psi \neq 0$, we conclude that ${\operatorname{span} \{ T_x M_y \varphi \}_{x,y \in G}} = \mathbb{F}_{q^2}^G$.
Finally, we have $\langle T_x M_y \varphi, T_x M_y \varphi \rangle = \langle \varphi, \varphi \rangle = a$ for every $x,y \in G$, since $T_x$ and $M_y$ are unitary.
Therefore $\Phi$ is an $(a,da)$-NTF.
\end{proof}

\subsection{Gerzon equality in finite unitary geometries}

As in the complex setting, translations and modulations can be used to create ETFs that achieve equality in Gerzon's bound.
The difficulty (as ever) lies in identifying a suitable fiducial vector~$\varphi$.
The next example shows how finite fields can simplify this problem by presenting a finite search space that retains many salient features of the complex setting.
(See Theorem~\ref{thm: Zauner} and Example~\ref{ex: SIC projections} further below for more examples of ETFs achieving Gerzon's bound.)

\begin{example}
We produce a $4\times 16$ Gabor ETF over a finite field.
Take $m=1$, $d_1 = d = 4$, and $q=31$.
Let $\zeta \in \mathbb{F}_{31^2}^\times$ be a primitive element, and define $\mu = \zeta^{120}$ and $\omega = \mu^2$.
The latter generate the unique subgroups of $\mathbb{T}_{31}$ with orders $8$ and $4$, respectively.
Define $R \in \mathbb{F}_{31^2}^{4\times 4}$ to be the diagonal matrix with entries $R_{jj} = \mu^{j(j+4)}$ for $j \in \mathbb{Z}/4\mathbb{Z}$, and let $\mathcal{F} = \begin{bmatrix} \omega^{ij} \end{bmatrix} \in \mathbb{F}_{31^2}^{4\times 4}$ be the DFT matrix of Definition~\ref{def: harmframe}.
(Throughout this example we use the ordering $\mathbb{Z}/4\mathbb{Z} = \{ 0,1,2,3\}$.)
Then $Z := R \mathcal{F} \in \mathbb{F}_{31^2}^{4 \times 4}$ is akin to (a scalar multiple of) Zauner's $4\times 4$ complex matrix, an eigenvector of which is known to generate a complex $4 \times 16$ ETF~\cite{Z99}.
A similar phenomenon occurs over $\mathbb{F}_{31^2}$, where the finite search space makes it easier to identify an appropriate fiducial vector.
Specifically, $Z$ has exactly three eigenvalues: $\lambda_1 = \zeta^{648}$, $\lambda_2 = \zeta^8$, and $\lambda_3 = \zeta^{328}$ with respective geometric multiplicities $2,1,1$.
The two-dimensional eigenspace for $\lambda_1$ contains exactly 924 one-dimensional subspaces.
Checking one representative from each line, we find (up to rescaling) exactly four fiducial vectors in the $\lambda_1$-eigenspace that generate Gabor ETFs, namely
\[
\varphi_1 = \left[ \begin{array}{l} 1 \\ \zeta \\ \zeta^{784} \\ \zeta^{70} \end{array} \right],
\quad
\varphi_2 = \left[ \begin{array}{l} 1 \\ \zeta^{70} \\ \zeta^{784} \\ \zeta \end{array} \right],
\quad
\varphi_3 = \left[ \begin{array}{l} 1 \\ \zeta^{391} \\ \zeta^{784} \\ \zeta^{610} \end{array} \right],
\quad
\varphi_4 = \left[ \begin{array}{l} 1 \\ \zeta^{610} \\ \zeta^{784} \\ \zeta^{391} \end{array} \right].
\]
~
\end{example}

\smallskip

We now identify fiducial vectors for infinitely many Gabor ETFs.
Theorem~\ref{thm: Zauner} proves that Gerzon's bound is attained in unitary geometries of infinitely many dimensions over the field $\mathbb{F} = \mathbb{F}_{3^2}$.
This is the first field for which this phenomenon is known to occur.
(In a companion paper we perform a similar feat for Gerzon's bound in orthogonal geometries~\cite{FFF2}.)

\begin{theorem}
\label{thm: Zauner}
Take $q = 3$, $m$ to be odd, and $d_k = 2$ for $1 \leq k \leq m$, so that $G = (\mathbb{Z}/2\mathbb{Z})^{m}$ and $d = 2^m$.
Let $\zeta \in \mathbb{F}_{3^2}$ be a primitive element, and define $\varphi \in \mathbb{F}_{3^2}^d$ by
\[
\varphi(x) =
\begin{cases}
-1 - \zeta^2, & \text{if } x = 0; \\
1, & \text{otherwise}.
\end{cases}
\]
Then $\Phi = \{ T_x M_y \varphi \}_{x,y \in G}$ is a $(0,1,0)$-ETF.
In particular, Gerzon's bound is attained in a unitary space over $\mathbb{F}_{3^2}$ whenever its dimension is twice a power of $4$.
\end{theorem}

\begin{proof}
Observe that $x = -x$ in $G$ and $2 \equiv -1$ in $\mathbb{F}_{3^2}$.
In particular, $d = 2^m \equiv -1$.
Furthermore $(\zeta^2)^2 = -1$ and $(\zeta^2)^3 = - \zeta^2$.

Let $\mathbf{1} \in \mathbb{F}_{3^2}^G$ be the function that is constantly~1, so  $\varphi = \mathbf{1} + {(1 - \zeta^2)} \delta_0$.
Here the coefficient on $\delta_0$ satisfies $(1 - \zeta^2)^3 = 1 + \zeta^2$ and $(1 - \zeta^2)^{3+1} = 1 - (\zeta^2)^2 = -1$.
With this in mind, for any $x,y \in G$ we expand to find
\[
\langle \varphi, T_x M_y \varphi \rangle
=
\langle \mathbf{1}, T_x M_y \mathbf{1} \rangle
+ (1 - \zeta^2) \langle \mathbf{1}, T_x M_y \delta_0 \rangle
+ (1 + \zeta^2) \langle \delta_0, T_x M_y \mathbf{1} \rangle
- \langle \delta_0, T_x M_y \delta_0 \rangle.
\]
Given $z \in G$ we have $(T_x M_y \mathbf{1})(z) = \hat{y}(z-x)$ and $(T_x M_y \delta_0)(z) = \delta_x(z)$.
Thus
\[
\langle \mathbf{1}, T_x M_y \mathbf{1} \rangle
= \sum_{z\in G} (T_x M_y \mathbf{1})(z)
= \sum_{z\in G} \hat{y}(z-x)
= \sum_{z\in G} \hat{y}(z)
= d \delta_{y,0}
= - \delta_{y,0}
\]
and
\[
\langle \mathbf{1}, T_x M_y \delta_0 \rangle
= \sum_{z\in G} (T_x M_y \delta_0)(z) = 1,
\]
while
\[
\langle \delta_0, T_x M_y \mathbf{1} \rangle
= (T_x M_y \mathbf{1})(0) 
= \hat{y}(-x)
= \hat{y}(x)
\]
and
\[
\langle \delta_0, T_x M_y \delta_0 \rangle
= (T_x M_y \delta_0)(0) = \delta_{x,0}.
\]
Therefore
\[
\langle \varphi, T_x M_y \varphi \rangle
=
- \delta_{y,0} 
+ 1 - \zeta^2 
+ (1 + \zeta^2)\hat{y}(x)
- \delta_{x,0}.
\]
Since $d_k = 2$ for every $k$ we must have $\hat{y}(x) \in \{ \pm 1 \}$.
Simplifying above, we find
\[
\langle \varphi, T_x M_y \varphi \rangle
=
\begin{cases}
0, & \text{if }x = y = 0; \\
1, & \text{if }\{0\} \subsetneq \{ x, y \}; \\
-1, & \text{if }0\notin \{ x, y\} \text{ and } \hat{y}(x) = 1; \\
\zeta^2, & \text{otherwise}.
\end{cases}
\]
In particular, $\langle \varphi, T_x M_y \varphi \rangle^{3 + 1} = 1$ whenever $\{x,y\} \neq \{ 0 \} $.
Now for any $s,t \in G$ with $(s,t) \neq (x,y)$ the relations~\eqref{eq: TM rels} produce
\[
\langle T_x M_y \varphi, T_s M_t \varphi \rangle
= \langle \varphi, M_{-y} T_{-x} T_s M_t \varphi \rangle
= \hat{y}(x - s) \langle \varphi, T_{s-x}M_{t-y} \varphi \rangle,
\]
so that $\langle T_x M_y \varphi, T_s M_t \varphi \rangle^{3+1} = 1$.
Hence $\Phi = \{ T_x M_y \varphi \}_{x,y \in G}$ is a $(0,1)$-equiangular system.
By Proposition~\ref{prop: Gabor tight frame} it is a~$(0,1,0)$-ETF.
\end{proof}

\begin{remark}
Over the complex numbers an ETF that attains equality in Gerzon's bound is also known as a \textbf{symmetric informationally complete positive operator valued measure}, or a \textbf{SIC-POVM}.
This terminology comes from quantum information theory, where such ETFs are important partly because their outer products $\varphi \varphi^\dagger$ provide a basis for operator space that consists of rank-one projections summing to a nonzero multiple of the identity~\cite{RBSC04}.

We intentionally avoid this terminology in Theorem~\ref{thm: Zauner} since the ETFs it produces consist of isotropic vectors, and the frame constant is zero.
Consequently the outer products all have trace zero, and no linear combination of them recreates the identity operator.
However, as shown in the proof of Gerzon's bound (Theorem~\ref{thm: Gerzon}) the outer products span the codimension-one space of traceless operators, and they form the finite field equivalent of a simplex in that space.
In doing so they produce the largest possible collection of traceless rank-one self-adjoint operators with constant pairwise value of $\langle \cdot, \cdot \rangle_F$.
\end{remark}

\begin{remark}
\label{rem: Hoggar}
The ETFs created in Theorem~\ref{thm: Zauner} may be seen as an infinite family that generalizes Hoggar's lines~\cite{H81,H98}.
The latter refers to an $8\times 64$ complex ETF with entries in the ring $\mathbb{Z}[i]$ of Gaussian integers.
(Here we have in mind the version given by Jedwab and Wiebe~\cite{JW15}.)
Specifically, let $G = (\mathbb{Z}/2\mathbb{Z})^3$, and consider the translation and modulation operations on $\mathbb{C}^G$ given by 
\[ (\tilde{T}_y \psi)(x) = \psi(x - y) \quad \text{and} \quad (\tilde{M}_y \psi)(x) = (-1)^{x \cdot y}  \qquad \text{for $\psi \in \mathbb{C}^G$, $x,y \in G$}, \]
where $x \cdot y$ denotes the dot product.
Define $\psi \in \mathbb{C}^G$ by 
\[ \psi(x) = \begin{cases} -1 + 2i, & \text{if }x = 0; \\ 1, & \text{otherwise.} \end{cases} \]
Then the system $\Psi = \{\tilde{T}_x \tilde{M}_y \psi \}_{x,y \in G}$ is an $8 \times 64$ complex ETF~\cite{JW15,SS16}.
To relate this with an ETF given by Theorem~\ref{thm: Zauner}, let $\zeta \in \mathbb{F}_{3^2}^\times$ be a primitive element, and let $f \colon \mathbb{Z}[i] \to \mathbb{F}_{3^2}$ be the unique ring homomorphism given by $f(a+bi) = a + b\zeta^2$.
(This is well defined since $\mathbb{Z}[i] \cong \mathbb{Z}[x] / ( x^2 + 1 )$ and $\zeta^4 = -1$.)
Observe that $f(\overline{z}) = f(z)^3$ for every $z \in \mathbb{Z}[i]$.
It follows that $f$ maps the ETF $\Psi \in \mathbb{Z}[i]^{8 \times 64}$ to an equiangular system $f(\Psi) \in \mathbb{F}_{3^2}^{8\times 64}$.
(Here we choose orderings on $G$ and $G \times G$ to identify $\mathbb{C}^G$ with $\mathbb{C}^8$, and so on.
We also extend $f$ to a mapping on matrices by entrywise application.)
In fact, $f(\Psi)= \Phi$ is exactly the ETF of Theorem~\ref{thm: Zauner} when $m=3$, where $f(\psi) = \varphi$ is the given fiducial vector in that case.

Furthermore, it is possible to recover Hoggar's lines from the finite field ETF $\Phi$ when $m=3$.
Explicitly, let $H \in \mathbb{F}_{q^2}^{64 \times 64}$ be the Gram matrix of $\Phi$.
As shown in the proof of Theorem~\ref{thm: Zauner}, $H$ has zeros on the diagonal, with off-diagonal entries in $\mathbb{T}_3 = \{ \pm 1, \pm \zeta^2 \}$.
Let $g \colon \{ 0 \} \cup \mathbb{T}_3 \to \mathbb{C}$ be the multiplicative character given by $g(0) = 0$ and $g(\zeta^{2l}) = i^l$.
Extending $g$ to a mapping on matrices, it turns out that $S := g(H) \in \mathbb{C}^{64 \times 64}$ has exactly two eigenvalues.
Adding an appropriate amount of identity, we find that $S + 3I$ is the Gram matrix of a complex $8\times 64$ ETF.
In fact $S+3I = \Psi^* \Psi$ is the Gram matrix of Hoggar's lines.

Sadly, this procedure does not produce a complex ETF of size $32 \times 1024$ when we take $m=5$ in Theorem~\ref{thm: Zauner}, where we found that the corresponding matrix $S \in \mathbb{C}^{1024 \times 1024}$ has three eigenvalues.
This is not surprising, since Godsil and Roy have shown that the group $G = (\mathbb{Z}/2\mathbb{Z})^m$ generates a complex $2^m \times 2^{2m}$ Gabor ETF only if $m \in \{1,3\}$~\cite{GR09}.

Finally, we remark that Hoggar's lines are highly symmetric and have a doubly transitive automorphism group~\cite{Zhu15,IM:DTI}.
We have not investigated the symmetries of the ETFs given by Theorem~\ref{thm: Zauner}, and we leave this problem for future study.
\end{remark}

\begin{problem}
For $\Phi$ as in Theorem~\ref{thm: Zauner}, determine the group $\operatorname{Aut} \Phi$ of all permutations $\mu \in S(G\times G)$ having components $\mu(x,y) =: \bigl(\mu_1 (x,y), \mu_2(x,y) \bigr)$ for which there exist scalars $c_\mu(x,y) \in \mathbb{F}_{q^2}^\times$ such that $T_{\mu_1(x,y)} M_{\mu_2(x,y)} \varphi = c_\mu(x,y) T_x M_y \varphi$ for every $x,y \in G$.
\end{problem}

\section{Finite field ETFs from complex ETFs}
\label{sec: C to F}
In this section we prove that every complex ETF produces ETFs in infinitely many finite fields.
In the process we show that the existence of a $d \times n$ complex ETF implies that of a $d \times n$ ETF with algebraic entries.
Our main result is the following.

\begin{theorem}
\label{thm: C to F}
Suppose there is a $d\times n$ complex ETF.
Then, for infinitely many pairwise coprime $q$, there is an ETF of $n$ vectors in a unitary geometry on $\mathbb{F}_{q^2}^d$.
\end{theorem}

We proceed in two steps.
First we nudge the complex ETF to have algebraic entries, then we map the algebraic ETF into infinitely many finite fields.

\subsection{Preliminaries}

First we recall some basic number theory, where standard references include~\cite{DF04,Ja73,Ne99}.
An \textbf{algebraic number} is a zero of a polynomial with rational coefficients; an \textbf{algebraic integer} is a zero of a monic polynomial with integer coefficients. 
The algebraic integers form a ring, and every algebraic number is a ratio of algebraic integers. 
A \textbf{number field} $E$ is a field with $\mathbb{Q} \leq E \leq \mathbb{C}$ and $\dim_{\mathbb{Q}} E < \infty$.
We write $\mathcal{O}_E$ for the ring of algebraic integers contained in $E$.
Its ideals have the following properties.
\begin{proposition}
\label{prop:idealbasics}
If $E$ is a number field, then the following hold for any proper nonzero ideal $\mathfrak{a} \subset \mathcal{O}_E$:
\begin{itemize}
\item[(a)]
$\mathcal{O}_E / \mathfrak{a}$ is finite,
\item[(b)]
if $\mathfrak{a}$ is prime then it is maximal,
\item[(c)]
only finitely many prime ideals of $\mathcal{O}_E$ contain $\mathfrak{a}$.
\end{itemize}
\end{proposition}
If $E$ is a number field, then there exists $\alpha \in E$ such that $E = \mathbb{Q}(\alpha)$, and $E = \{ f(\alpha) : f \in \mathbb{Q}[x] \}$.
The \textbf{minimal polynomial} of $\alpha$ is the monic polynomial $m_\alpha \in \mathbb{Q}[x]$ of lowest degree such that $m_\alpha(\alpha) = 0$.
We say $E$ is \textbf{Galois} if it contains every root of $m_\alpha$.
In any case, there exists a Galois number field containing $E$.

If $E$ is Galois then both $E$ and $\mathcal{O}_E$ are closed under complex conjugation.
This gives each the structure of a \textbf{$\ast$-ring}, that is, a ring equipped with an involutory ring automorphism.
A \textbf{$\ast$-ring homomorphism} is a ring homomorphism between $\ast$-rings that preserves the involution.
This completes our brief review.

\subsection{Complex ETFs with algebraic entries}

Now we show that the existence of a $d \times n$ complex ETF implies that of a $d \times n$ ETF with algebraic entries.

\begin{theorem}
\label{thm:alg}
If there is a $d\times n$ complex ETF, then there is a $d\times n$ complex ETF with algebraic entries.
\end{theorem}

To prove Theorem~\ref{thm:alg}, consider the algebra of sets generated by sets of the form
\[
\{x\in\mathbb{R}^n:f(x)\geq0\},
\qquad
f\in\mathbb{Z}[x_1,\ldots,x_n].
\]
We refer to members of this algebra as \textbf{integral semialgebraic sets}.
For example, by taking real and imaginary parts of matrix entries, the set of $d\times n$ complex ETFs may be viewed as an integral semialgebraic subset of $\mathbb{R}^{2dn}$.
As such, Theorem~\ref{thm:alg} is a special case of the following.

\begin{lemma}
Every nonempty closed integral semialgebraic set contains a point whose coordinates are all real algebraic numbers.
\end{lemma}

\begin{proof}
Let $S\subseteq\mathbb{R}^n$ be nonempty, closed, and integral semialgebraic.
By nonemptyness, there exists an integer $k$ such that $S_0:=S\cap\{x:\|x\|^2\leq k\}$ is nonempty.
For each $i\in[n]$, let $\pi_i\colon\mathbb{R}^n\to\mathbb{R}$ denote projection onto the $i$th coordinate.
We will iteratively take 
\[
z_i:=\sup\pi_i(S_{i-1}),
\qquad
S_i:=S_{i-1}\cap\{x \in \mathbb{R}^n : \pi_i(x)=z_i\}.
\]
We claim that the following hold for each $i\in[n]$:
\begin{itemize}
\item[(i)]
$z_i$ is a real algebraic number,
\item[(ii)]
$S_i$ is a nonempty compact integral semialgebraic subset of $\mathbb{R}^n$.
\end{itemize}
Considering $S\supseteq S_n=\{(z_1,\ldots,z_n)\}$, the result follows from this claim.

Observe that $S_0$ is a nonempty compact integral semialgebraic subset of $\mathbb{R}^n$.
We will show that (i) and (ii) together follow from $S_{i-1}$ being a nonempty compact integral semialgebraic subset of $\mathbb{R}^n$, and then the claim follows by induction.
Since $S_{i-1}$ is nonempty, compact, and integral semialgebraic, it follows from Tarski--Seidenberg (Theorem~1.4.2 of~\cite{BCR98}) that $\pi_i(S_{i-1})$ is also nonempty, compact, and integral semialgebraic.
As such, $\pi_i(S_{i-1})$ is the disjoint union of finitely many compact intervals with real algebraic endpoints.
It follows that (i) holds.
Since $z_i\in\pi_i(S_{i-1})$, it follows that $S_i$ is nonempty.
To finish the proof of (ii), considering our hypothesis on $S_{i-1}$, it suffices to demonstrate that $\{x:x_i=z_i\}$ is closed and integral semialgebraic.
It is closed since $\pi_i$ is continuous.
To see it is integral semialgebraic, consider any polynomial $f\in\mathbb{Z}[x]$ for which $f(z_i)=0$, and select $a,b\in\mathbb{Q}$ such that $z_i$ is the only root of $f$ in $[a,b]$.
Then 
\[
\{x:x_i=z_i\}
=\{x:f(x_i)=0\}\cap\{x:x_i\geq a\}\cap\{x:x_i\leq b\}
\]
is integral semialgebraic, as desired.
\end{proof}

\subsection{Projecting complex ETFs with algebraic entries}

Next we show that the existence of a complex ETF implies that of corresponding ETFs over infinitely many finite fields.

\begin{theorem}
\label{thm:project}
Suppose $\Phi \in \mathbb{C}^{d\times n}$ is an ETF with entries in $\mathcal{O}_E$, where $E$ is a Galois number field.
Then, for infinitely many pairwise coprime $q$ there exists a $\ast$-ring homomorphism $f \colon \mathcal{O}_E \to \mathbb{F}_{q^2}$ such that $f(\Phi) \in \mathbb{F}_{q^2}^{d\times n}$ is a $d\times n$ ETF in the complex model.
\end{theorem}

Our proof of Theorem~\ref{thm:project} uses the following.

\begin{lemma}
\label{lem: nice ideals}
Let $E$ be a Galois number field. 
Given any choice of nonzero algebraic integers $z_1,\dotsc,z_n \in E$, there exist infinitely many ideals $\mathfrak{p} \subset \mathcal{O}_E$ such that all of the following are true:
\begin{itemize}
\item[(i)]
$\mathcal{O}_E/\mathfrak{p}$ is a finite field,
\item[(ii)]
$\mathfrak{p}$ is closed under complex conjugation,
\item[(iii)]
$\mathfrak{p}$ does not contain any of $z_1,\dotsc,z_n$.
\end{itemize}
Additionally, the set of field characteristics
\[
\{ \operatorname{char} \mathcal{O}_E/ \mathfrak{p} : \mathfrak{p} \subset \mathcal{O}_E \text{ is an ideal satisfying (i)--(iii) } \}
\]
so obtained is infinite.
\end{lemma}

\begin{proof}
Let $F = E \cap \mathbb{R}$.
Then $F$ is the fixed field for the subgroup of $\operatorname{Gal}(E/\mathbb{Q})$ generated by complex conjugation. 
By the Fundamental Theorem of Galois Theory, $E$ is a Galois extension of $F$ with automorphism group $\operatorname{Gal}(E/F) \cong \mathbb{Z}/2\mathbb{Z}$.
In particular, $E$ is a cyclic extension of $F$.

As a consequence of the Frobenius density theorem (Cor.~5.4 in \cite[Ch.~V]{Ja73}), there are infinitely many prime ideals $\mathfrak{q} \subset \mathcal{O}_F$ for which $\mathfrak{p} = \mathfrak{q} \mathcal{O}_E$ is a prime ideal of $\mathcal{O}_E$.
Furthermore, for each $\mathfrak{p} \subset \mathcal{O}_E$ there is at most one prime ideal $\mathfrak{q} \subset \mathcal{O}_F$ such that $\mathfrak{p} = \mathfrak{q} \mathcal{O}_E$, namely $\mathfrak{q} = \mathfrak{p} \cap \mathcal{O}_F$.
(Indeed, if $\mathfrak{q} \subset \mathcal{O}_F$ is prime and $\mathfrak{p} = \mathfrak{q} \mathcal{O}_E$, then $\mathfrak{p} \cap \mathcal{O}_F$ contains the maximal ideal $\mathfrak{q}$.)
Hence there are infinitely many prime ideals $\mathfrak{p} \subset \mathcal{O}_E$ of the form $\mathfrak{p}=\mathfrak{q} \mathcal{O}_E$, where $\mathfrak{q}$ is a prime ideal of $\mathcal{O}_F$.
Moreover, there are only finitely many prime ideals in $\mathcal{O}_E$ that contain any of the ideals $z_1\mathcal{O}_E,\dotsc,z_n\mathcal{O}_E$.
Overall, there are infinitely many choices of prime ideals $\mathfrak{p} \subset \mathcal{O}_E$ that avoid  $z_1,\dotsc,z_n$ and take the form $\mathfrak{p}= \mathfrak{q} \mathcal{O}_E$, where $\mathfrak{q} \subset \mathbb{R}$.
Any such $\mathfrak{p}$ is closed under complex conjugation since this is true of both $\mathfrak{q}$ and $\mathcal{O}_E$, and $\mathcal{O}/\mathfrak{p}$ is a finite field by Proposition~\ref{prop:idealbasics}.

To prove the ``additionally'' statement, we verify that any prime ideal $\mathfrak{p}$ contains $(\operatorname{char} \mathcal{O}_E/\mathfrak{p}) \mathcal{O}_E$.
By considering the minimal polynomial of a nonzero element of $\mathfrak{p}$, we see that $\mathfrak{p}$ contains a nonzero integer (namely the opposite of the polynomial's constant term).
As such, there exists a positive prime $p\in \mathbb{Z}$ such that $p \mathbb{Z} = \mathfrak{p} \cap \mathbb{Z}$, the latter being a nonzero prime ideal of $\mathbb{Z}$.
We have $p \mathcal{O}_E \subset \mathfrak{p}$, and the containment
\[
\mathcal{O}_E / \mathfrak{p} \supset (\mathbb{Z} + \mathfrak{p})/\mathfrak{p} \cong \mathbb{Z} / (\mathbb{Z} \cap \mathfrak{p}) \cong \mathbb{Z}/p\mathbb{Z},
\]
demonstrates that $p = \operatorname{char} \mathcal{O}_E / \mathfrak{p}$.
This proves the claim.
By Proposition~\ref{prop:idealbasics}(c), there are only finitely many prime ideals $\mathfrak{p} \subset \mathcal{O}_E$ associated with any given characteristic $p = \operatorname{char} \mathcal{O}_E / \mathfrak{p}$.
Since infinitely many prime ideals satisfy (i)--(iii), the set of associated field characteristics must also be infinite.
\end{proof}

\begin{proof}[Proof of Theorem~\ref{thm:project}]
First observe that the frame constant $c$ for $\Phi$ belongs to $\mathcal{O}_E$ since $cI = \Phi \Phi^*$.
Select an ideal $\mathfrak{p} \subset \mathcal{O}_E$ as in Lemma~\ref{lem: nice ideals}, where $\mathfrak{p}$ does not contain $c$. 
Denote $K = \mathcal{O}_E/\mathfrak{p}$, and let $g \colon \mathcal{O}_E \to K$ be the quotient mapping. 
Since $\mathfrak{p}$ is closed under complex conjugation, there is a well-defined field automorphism $\sigma$ of $K$ given by $g(x)^\sigma = g(\overline{x})$ for $x \in \mathcal{O}_E$, and $\sigma^2 = 1$. 
By passing to a quadratic extension of $K$ if necessary, we obtain a finite field $\mathbb{F}_{q^2}$ and a ring homomorphism $f \colon \mathcal{O}_E \to \mathbb{F}_{q^2}$ with kernel $\mathfrak{p}$, such that $f(\overline{x}) = x^q$ for every $x \in \mathcal{O}_E$.
Then $f(\Phi)f(\Phi)^* = f(\Phi \Phi^*) =  f(c)I \neq 0$, and so $f(\Phi)$ is a tight frame for $\mathbb{F}_{q^2}^d$.
The other properties are verified similarly, and $f(\Phi)$ is an ETF.
\end{proof}

Finally, we obtain our main result as a corollary of Theorem~\ref{thm:project}.

\begin{proof}[Proof of Theorem~\ref{thm: C to F}]
Suppose there is an ETF $\Phi \in \mathbb{C}^{d\times n}$.
By Theorem~\ref{thm:alg} we may take $\Phi$ to have algebraic entries.
After rescaling $\Phi$ to clear any algebraic integer denominators, we may assume its entries are in fact algebraic integers.
Now let $E$ be any Galois extension of $\mathbb{Q}$ containing the entries of $\Phi$, and apply Theorem~\ref{thm:project}.
\end{proof}

\begin{remark}
\label{rem: project Gram}
Many ETFs $\Phi \in \mathbb{C}^{d\times n}$ are constructed by way of their Gram matrices $G = \Phi^* \Phi$.
Often $G$ carries a nice structure while $\Phi$ is essentially unknown (yet guaranteed to exist by Cholesky decomposition).
In such cases it may be desirable to project $G$ into a finite field instead of $\Phi$ itself.
The procedure in Theorem~\ref{thm:project} works just as well to map $G$ into finite fields with infinitely many characteristics, provided the entries of $G$ are all algebraic integers.
Indeed, let $E$ be a number field whose integer ring $\mathcal{O}_E$ contains the entries of $G$ as well as its nonzero eigenvalue $c$.
Choose any ideal $\mathfrak{p} \subset \mathcal{O}_E$ not containing $c$ as in Lemma~\ref{lem: nice ideals}, and let $f \colon \mathcal{O}_E \to \mathbb{F}_{q^2}$ be a $\ast$-ring homomorphism with kernel $\mathfrak{p}$.
Then $f(G) \in \mathbb{F}_{q^2}^{n\times n}$ is self-adjoint, and $d' := \operatorname{rank} f(G) \leq \operatorname{rank} G = d$ since the rank of a matrix is the largest size of a square submatrix having nonzero determinant.
Moreover, applying $f$ to the coefficients of the characteristic polynomial of $G$ shows that $\operatorname{det} \bigl( xI - f(G) \bigr) = [x - f(c)]^d x^{n-d}$.
Considering the Jordan normal form of $f(G)$ we conclude that $d' \geq d$, and equality holds.
Applying Theorem~\ref{thm: factor U}, we conclude that $f(G)$ factors to produce a $d \times n$ ETF in a unitary geometry over $\mathbb{F}_{q^2}$.
\end{remark}

\begin{example}
Let $k \geq 2$ be a positive integer, and let $\alpha_k \in \mathbb{C}$ be a primitive $k$-th root of unity.
Suppose $\Phi \in \mathbb{C}^{d\times n}$ is an ETF whose frame constant $c$ and Gram matrix entries all lie in $\mathbb{Z}[\alpha_k]$.
Then we can project $\Phi$ into the complex model over a finite field as follows.
Choose any prime power $q$ such that $k$ divides $q+1$, and let $\omega_k \in \mathbb{T}_{q} \leq \mathbb{F}_{q^2}^\times$ be a generator for the unique subgroup of order $k$.
Denote the $k$-th cyclotomic polynomial by $m_k(x) \in \mathbb{Z}[x]$, and recall that $x^k - 1 = \prod_{j \mid k} m_j(x)$.
Then $\mathbb{Z}[\alpha_k] \cong \mathbb{Z}[x] /  \bigl(m_k(x) \mathbb{Z}[x] \bigr)$, and $m_k(\omega_k) = 0$ since $\omega_k$ is a root of $x^k - 1$ but not of $x^j - 1$ for $j < k$.
Consequently, there is a well-defined ring homomorphism $f \colon \mathbb{Z}[\alpha_k] \to \mathbb{F}_{q^2}$ given by $f(g(\alpha_k)) = g(\omega_k)$ for every $g \in \mathbb{Z}[x]$.
Furthermore, $f(\overline{z}) = f(z)^q$ for every $z \in \mathbb{Z}[\alpha_k]$ since $\overline{\alpha_k} = \alpha_k^{k-1}$ and $\omega_k^{k-1} = \omega_k^q$.
It follows easily that $f(\Phi^* \Phi) \in \mathbb{F}_{q^2}^{n\times n}$ is the Gram matrix of an ETF in the complex model on $\mathbb{F}_{q^2}^{d'}$, where $d' = \operatorname{rank} f(\Phi^* \Phi)$.
We have $d' \leq d$ with equality if $f(c) \neq 0$, as in Remark~\ref{rem: project Gram}.
Finally, if $\Phi \in \mathbb{Z}[\alpha_k]^{d \times n}$ and $f(c) \neq 0$ then $f(\Phi) \in \mathbb{F}_{q^2}^{d\times n}$ is itself an ETF, as in the proof of Theorem~\ref{thm:project}.
\end{example}

\begin{example}
\label{ex: SIC projections}
Table~\ref{tbl: SIC projections} describes some finite unitary geometries that admit Gabor ETFs, where the abelian groups that provide translations and modulations are cyclic.
In creating the table we began with a known fiducial vector for a complex Gabor ETF, and then applied the construction of Theorem~\ref{thm:project} to map it into a finite field.
The fiducial vector $\varphi \in \mathbb{C}^d$ was taken from one of~\cite{Gr,SG10,ACFW18}.
In each case the resulting Gabor frame $\Phi \in \mathbb{C}^{d\times d^2}$ had entries in a number field $E$, and we multiplied by a scalar to clear fractions and put the entries in $\mathcal{O}_E$.
Defining $F = E \cap \mathbb{R}$, we then used Magma~\cite{magma} to identify a prime ideal $\mathfrak{q} \subset \mathcal{O}_F$ for which $\mathfrak{p}:=\mathfrak{q} \mathcal{O}_E$ remained a prime ideal of $\mathcal{O}_E$.
Table~\ref{tbl: SIC projections} gives the size of $\mathcal{O}_E / \mathfrak{p} \cong \mathbb{F}_{q^2}$.
Here, the quotient mapping produces a $\ast$-ring homomorphism $f \colon \mathcal{O}_E \to \mathbb{F}_{q^2}$, and $f(\Phi) \in \mathbb{F}_{q^2}^{d\times d^2}$ is an ETF as in Theorem~\ref{thm:project}.
Furthermore, considering the images under $f$ of complex translation and modulation matrices, we see that $f(\Phi)$ is itself a Gabor ETF.

We emphasize that Table~\ref{tbl: SIC projections} does not describe all finite fields that admit projections of the given complex fiducial vectors.
Our method focused on ideals of $\mathcal{O}_E$ that lie entirely in $\mathbb{R}$, but in many cases there are other prime ideals of $\mathcal{O}_E$ that are closed under complex conjugation, and the corresponding quotient mappings yield Gabor ETFs in other finite fields not listed in Table~\ref{tbl: SIC projections}.
\end{example}

\begin{table}
\begin{tabular}{c|l}
        $\varphi$ &  $q$  \\
        \hline
        2a & $167$, $191$, $239$, $263$, $311$, $383$, $743$, $863$, $887$, $911$, $983$, $1031$, $1103$ \\
        3a, 3b, 3c &  $167$, $191$, $239$, $263$, $311$, $383$, $743$, $863$, $887$, $911$, $983$, $1031$, $1103$ \\
        4a & $71$, $191$, $239$, $311$, $359$, $431$, $479$, $599$, $719$, $839$, $911$, $1031$, $1151$\\%
        5a & $179$, $239$, $359$, $419$, $479$, $599$, $659$, $719$, $839$, $1019$, $1259$, $1319$ \\
        6a & $47^3$, $59^3$, $83^3$, $131$, $167^3$, $227$, $251^3$, $311^3$, $383^3$, $419^3$, $467$, $479^3$ \\
        7a, 7b & $31^3$, $47^3$, $103^3$, $167$, $199^3$, $223$, $271^3$, $311^3$, $367^3$, $383^3$, $439^3$, $479^3$ \\
        8a & $479$, $911$, $2351$, $2399$, $2591$, $2879$, $3119$,  $5279$, $5471$, $5711$, $6959$ \\
        8b & $79$, $191$, $239$, $271$, $431$, $479$, $719$, $751$, $911$, $991$, $1039$, $1151$ \\
        9a, 9b & $11^3$, $59^3$, $71^3$, $131^3$, $179^3$, $191^3$, $239^3$, $251^3$, $311^3$, $359^3$, $419^3$, $431$ \\
        10a & $19^3$, $479^3$, $1319^3$, $1559$, $1979^3$, $2939$, $2999^3$, $3659^3$, $3779^3$, $4259^3$\\ 
        11a, 11b & $239^5$, $2879^5$, $5519^5$, $10559$, $11519^5$, $12239^5$, $14159^5$, $14519$  \\ 
        11c & $167^5$, $239^5$, $1223^5$, $1487^5$, $2063^5$, $2111$, $2207^5$, $2543^5$, $2591^5$, $2879^5$\\
        12a &  $263^3$, $503^3$, $599^3$, $647^3$, $1439^3$, $1871^3$, $2063^3$, $2207$, $2447^3$, $2591$  \\
        12b & $263$, $503$, $599$, $647$, $1439$, $1871$, $2063$, $2207$, $2447$, $2591$, $2687$  \\ 
        13a, 13b &  $251^3$, $439^3$, $1291^3$, $1559$, $2539^3$, $3631^3$, $4339^3$ $4679$, $5431^3$, $5659^3$ \\
        14a, 14b &  $131^3$, $479^3$, $1091^3$, $1151^3$, $1319^3$, $1559^3$, $1811^3$, $1931^3$, $1979^3$  \\
        15a, 15b, 15c & $239^3$, $359^3$, $719^3$, $1319^3$, $2879$, $2999^3$, $3359^3$, $4799^3$, $4919$, $5039^3$ \\
        15d &  $59^3$, $179^3$, $239^3$, $359^3$, $419$, $479$, $659$, $719^3$, $839^3$, $1019$, $1259^3$
    \end{tabular}

\medskip
\caption{
The first column describes a known fiducial vector for a complex Gabor ETF, as in~\cite{SG10, ACFW18, Gr}.
The number $d$ in front of the letter gives the dimension of the ETF, and $\mathbb{Z}/d\mathbb{Z}$ is the abelian group that provides translations and modulations.
The second column lists some prime powers $q$ for which the complex ETF may be projected into $\mathbb{F}_{q^2}^{d \times d^2}$ as a Gabor ETF over a finite field.
The finite field fiducial vectors are included in an ancillary file with the arXiv version of this paper.
Example~\ref{ex: SIC projections} gives our methodology.
}
\label{tbl: SIC projections}
\end{table}

\subsection{Open problems}

We end with problems for future investigation.
The first is the finite field analog of Zauner's conjecture (Conjecture~\ref{conj: Zauner complex}).

\begin{conjecture}[Zauner's conjecture over finite fields]
\label{conj: finite field Zauner}
For every $d$, there exist infinitely many pairwise coprime $q$ such that a unitary geometry on $\mathbb{F}_{q^2}$ admits an ETF of $n = d^2$ vectors.
\end{conjecture}

According to Theorem~\ref{thm: C to F}, if Zauner's conjecture holds over the complex numbers then Conjecture~\ref{conj: finite field Zauner} is also satisfied.
More generally, we pose the following.

\begin{problem}
For which $(d,q)$ does a unitary geometry on $\mathbb{F}_{q^2}^d$ admit an ETF of $n = d^2$ vectors?
\end{problem}

Next, the converse of Theorem~\ref{thm: C to F} remains open.

\begin{problem}
\label{prob: F to C}
If there exist $d\times n$ ETFs in unitary geometries over infinitely many finite fields with distinct characteristics, does there also exist a $d\times n$ complex ETF?
\end{problem}

If the words ``unitary'' and ``complex'' are replaced by ``orthogonal'' and ``real'' in Problem~\ref{prob: F to C}, then the resulting question has an affirmative answer.
See Proposition~3.3 of~\cite{FFF2}.

\begin{problem}
\label{prob: nec conds}
Identify necessary conditions for the existence of $d\times n$ ETFs in finite unitary geometries.
\end{problem}

An effective solution to Problem~\ref{prob: nec conds}, when paired with Theorem~\ref{thm: C to F}, would provide necessary conditions on the existence of complex ETFs.

\section*{Acknowledgments}
GRWG was partially supported by the Singapore Ministry of Education Academic Research Fund (Tier 1); grant numbers: RG29/18 and RG21/20.
JJ was supported by NSF DMS 1830066.
DGM was partially supported by AFOSR FA9550-18-1-0107 and NSF DMS 1829955.
This project began at the 2018 MFO Mini-Workshop on Algebraic, Geometric, and Combinatorial Methods in Frame Theory.
The authors thank the other participants, Matt Fickus, and Steve Flammia for comments that motivated the authors or provided insight.

\bigskip

\bibliographystyle{abbrv}
\bibliography{refsff1}

\end{document}